\date{January 25, 2013}
\documentclass[12pt]{amsart}
\usepackage{latexsym,amsmath,amsfonts,amscd,amssymb}
\usepackage{graphics}
\newtheorem{theorem}{Theorem}[section]
\newtheorem{lemma}[theorem]{Lemma}
\newtheorem{proposition}[theorem]{Proposition}
\newtheorem{corollary}[theorem]{Corollary}

\newtheorem{definition}[theorem]{Definition}

\theoremstyle{remark}
\newtheorem{remark}[theorem]{Remark}

\newcommand{\Capa}{\operatorname{Cap}}

\newcommand{\diam}{\operatorname{diam}}

\newcommand{\grad}{\operatorname{grad}}

\newcommand{\Homeo}{\operatorname{Homeo}}

\newcommand{\id}{\operatorname{id}}

\newcommand{\supp}{\operatorname{supp}}

\newcommand{\cC}{{\mathcal C}}

\newcommand{\cG}{{\mathcal G}}

\newcommand{\cM}{{\mathcal M}}

\newcommand{\cS}{{\mathcal S}}
\newcommand{\cT}{{\mathcal T}}

\newcommand{\CC}{{\mathbb C}}
\newcommand{\DD}{{\mathbb D}}
\newcommand{\HH}{{\mathbb H}}

\newcommand{\RR}{{\mathbb R}}
\newcommand{\TT}{{\mathbb T}}

\newcommand{\ZZ}{{\mathbb Z}}

\renewcommand{\a}{\alpha}

\renewcommand{\d}{\delta}
\newcommand{\g}{\gamma}

\newcommand{\eps}{\epsilon}

\newcommand{\s}{\sigma}

\renewcommand{\o}{\omega}

\renewcommand{\t}{\tau}

\newcommand{\G}{\Gamma}
\renewcommand{\O}{\Omega}

\title{Degenerate conformal structures}


\author[R. P\'{e}rez Marco]{Ricardo P\'{e}rez Marco}
\address{CNRS, LAGA UMR 7539, Universit\'e Paris XIII,
99, Avenue J.-B. Cl\'ement, 93430-Villetaneuse, France}
\email{ricardo.perez.marco@gmail.com}

\begin{document}

\begin{abstract}
We present new rectification theorems
of degenerate quasi-conformal structures that give a meaning to 
quotients of Riemann surfaces with empty interior 
``fundamental domains''. These techniques are used to define the unique renormalization of polynomials with Cantor set Julia sets.
\end{abstract}

\maketitle

\noindent \emph{The spirit of Riemann will move future 
generations as it has moved us (L.V. Ahlfors)}

%\newpage

%\setcounter{tocdepth}{2}
\tableofcontents

%\newpage

\section{Introduction}\label{intro}

A quotient of a topological space by an arbitrary equivalent relation
has a natural topological space structure. 
The quotient of a Hausdorff topological 
space by a Hausdorff equivalence relation is a Hausdorff
topological space. Manifolds are defined as Hausdorff
topological spaces endowed with a {\it local} smooth 
structure determined by an atlas.
Therefore, classically, in order to quotient manifolds we need an equivalence relation 
with some non-local structure: {\it A fundamental domain}.

%\medskip

Even when the quotient is a topological manifold, without a fundamental domain, it
seems impossible from the classical point of view 
to recover a natural smooth 
structure inherited from the original one. 
The main purpose of this article is to overcome this difficulty in natural specific situations \textit{for 
a complex structure}. We show that some 
natural Hausdorff quotients of the Riemann sphere 
without fundamental domain led to a topological two dimensional sphere with a 
canonical complex structure. Some of these quotients are dynamically motivated and natural, but not all of them. 
In these quotients of the Riemann  
sphere, an open dense set of 
total measure can be collapsed into a set of zero measure. More 
precisely, the topological  
quotient is determined by an infinite number of holes
in the Riemann sphere. These holes have piecewise analytic 
boundaries and they can be dense in the sphere.
The quotient is then obtained by pasting analytically the
boundaries of the holes in this "butchered  Riemann sphere". 
Similar quotients exist in  
general Riemann surfaces, but we restrict to consider the Riemann sphere in this article that contains all the analytic complexity.

\medskip

\begin{figure}[h] 
\centering
\resizebox{8cm}{!}{\includegraphics{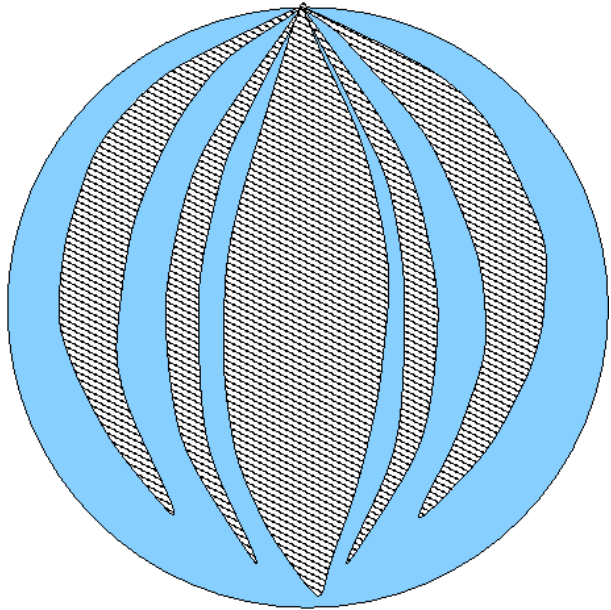}}    % name of the file - without extension
\caption{The butchered Riemann sphere.}
\end{figure}

\medskip
One of the main heuristic ideas behind these techniques
is the belief that complex structures 
can be though of as a more basic structure that  the 
topological one. This seems paradoxical. A classical complex structure
determines an underlying topological structure. The heuristics 
is that the classical definition is only a 
restricted definition of complex structures to topological manifolds.
The general unrestricted definition has still to be found. Many results in 
dynamics and geometry hint that such a general definition may exist. This philosophy 
is very useful in Holomorphic Dynamics, in particular in Hedgehog Theory (hedgehogs
are topologically complex invariant compacts associated to indifferent 
periodic orbits discovered by the author \cite{[PM1]} whose dynamics exhibits remarkable rigidity properties that are only to be expected when one thinks 
to the hedgehog as possesing an intrinsic complex structure). 
The general definition should be well behaved, i.e. immersions
will induce complex structures. Also should enjoy compactness 
properties as those of Hausdorff-Gromov length spaces.
This circle of ideas can be traced back to Riemann where the 
fundamental entity of complex structures permeates everywhere 
in the "Inaugural Dissertation".

\bigskip

This philosophy is a valuable guide in our problem of quotienting
the butchered Riemann sphere. We have a compact set in the 
sphere $K$ (the complement of the open holes), that plays the 
role of ``fundamental domain'' (but may have empty interior...)
The complex structure on the Riemann sphere should induce an 
intrinsic complex structure on $K$ that will induce a complex
structure on the quotiented sphere. So one might expect to have
a natural complex structure in the quotient. Thus also a natural
absolutely continuous class in the quotiented topological sphere.
Note that the absolutely continuous class will be determined by 
the conformal structure (and not the other way around as classically!). 
There seems to be a priori no natural way
to recover directly the absolutely continuous class without going through the complex structure.

\bigskip

Because we don't have a general definition of intrinsic complex
structures, we will construct the complex structure in 
the quotient in an indirect way. 
In dimension $2$ complex structures can be identified with  
conformal structures. One of the main ideas consists in getting 
the wild topological quotient as a limit of quasi-conformal 
deformations which degenerate into the quotient. Necessarily the 
quasi-conformal distorsion blows out since at the limit we get
topological collapsing. These quasi-conformal deformations are 
defined by Beltrami coefficients that can degenerate everywhere 
(i.e. 
$||\mu_n ||_{L^\infty }\to 1$ locally everywhere.) 
Classical quasi-conformal theory 
is unable to provide the existence of a limit map of these rectification mappings 
for this wild type of collapsing.
The second fundamental ingredient consists in combining classical 
quasi-conformal theory with classical potential theory. Roughly speaking
the quasi-conformal deformations will be compatible with the potential 
of a compact set in the Riemann sphere (the holes in the 
butchered Riemann 
sphere are Green lines for this potential). Compactness coming from 
classical potential theory shows the existence of a unique 
topological collapsing that is the limit of the quasi-conformal 
deformations.
The mixture of classical quasi-conformal theory and classical 
potential theory in order to obtain unique limits for degenerating
quasi-conformal mappings is the new ingredient that provides the compactness
that is lacking in the classical theory.

D. Sullivan already
pointed out the analogy between Beltrami coefficients parametrizing
quasi-conformal mappings and $L^\infty$ densities parametrizing 
Lipschitz mappings (see \cite{[Su]} p.468)
Pushing this analogy one might expect that degenerate 
Beltrami coefficients will parametrize non-quasiconformal 
homeomorphisms in the same way that 
non-atomic measures parametrize 
general homeorphisms.
The notion of virtual conformal structures (see section  \ref{conf_struct})
explains very accurately this analogy.

\medskip

\textbf {New rectifications theorems.}

\medskip

The first results in section \ref{rectification} consist on new rectification 
theorems for degenerate conformal structures. In this article a
``conformal structure'' in a domain $\O $ of the Riemann sphere is defined 
in a generalized form: It consists of  
a field of conformal classes in tangent spaces
at almost all points in $\O $. No assumption of quasi-conformality
with respect to the standard quasi-conformal structure is made.
Conformal structures can be transported by absolutely continuous
almost everywhere differentiable homeomorphisms, or even continuous maps. This is a more 
natural category than the one of quasi-conformal homeomorphisms. The natural
objection we can think of is the lack of rectification theorems.
We prove different sorts of rectification theorems for degenerate quasi-conformal
conformal structures (we label them ``degenerate'' to stress that they are not quasi-conformal).  
First we start considering conformal structures
that are quasi-conformal on compact subsets of the complement 
of a given compact, then potential conformal structures and finally 
virtual conformal structures. The definitions of this notions and
necessary background is contained in section \ref{conf_struct}.

Without more notations we can already state a first rectification 
theorem. Consider a compact set $K$ in the Riemann sphere 
${\overline {\CC}}$ and a conformal structure $\xi$ in ${\overline 
{\CC}} -K$ which is locally quasi-conformal in ${\overline {\CC}}-K$.
Let $\mu$ be the Beltrami coefficient of $\xi$ ($\mu =0$ on $K$).

\begin{theorem}{If the Riemann surface $({\overline {\CC}}-K , \xi )$ is 
in the class $O_{AD}$ then, up to a Mo\" ebius composition to 
the left,
there exists a unique continuous
map $h : {\overline {\CC}} \to {\overline {\CC}}$ such 
that $h$ is locally quasi-conformal on ${\overline {\CC}}-K$,
and on ${\overline {\CC}}-K$,
$$
\bar \partial h =\mu \ \partial h .
$$
}
\end{theorem}

Note that the rectification $h$ in this theorem is not even a 
homeomorphism in general.

The Riemann surface $({\overline {\CC}}-K , \xi )$ is the 
Riemann surface structure induced by the locally quasi-conformal 
structure $\xi$ (hence $\xi$ defines locally a complex structure 
by classical quasi-conformal theory). A Riemann surface is in the class $O_{AD}$ if
there no non-constant holomorphic functions with finite 
Dirichlet integrals. There is a geometric criterium due to 
Sario that can be used to recognize when a Riemann surface is 
in the class $O_{AD}$. This is exploited to give effective 
versions of this rectification theorems and to  prove the 
rectifications theorems for potential conformal structures in
section \ref{rectif_pot} and \ref{rectif_virt}.

The rectification theorem for virtual conformal structures is 
given in section \ref{rectif_virt}. The notion of solution to the 
Beltrami equation has to be reinterpreted. We obtain 
"rectifications" that are totally degenerate and realize the 
collapsing of the butchered Riemann sphere into a new
Riemann sphere.

In section \ref{continua} we introduce generalized rectifications of a 
different kind. We use the potential theory outside a 
connected compact set and make use of Rickman's removability 
theorem. These rectifications are useful
in the new theory of renormalization of polynomials with 
connected Julia sets.

The motivation for this article comes from problems in Holomorphic
Dynamics and Renormalization Theory. Here we present some of the analytic 
tools that are necessary to a new approach to the 
renormalization of polynomials.  The main interest of this new renormalization 
is that it is canonical and, contrary to the classical Douady-Hubbard approach, natural 
in the class of polynomials since it 
remembers the polynomial structure of the dynamics by using the combinatorical structure 
of external rays (and not 
only the polynomial-like structure). In particular, it allows to define uniquely the renormalization 
of polynomial combinatorically renormalizable and with Cantor set Julia set for the quadratic polynomials (and for higher degree
if one considers removable Julia sets) (see \cite{[PM2]}).

\bigskip

{\bf Acknowledgements.}

I am grateful to John Garnett who read the first versions of
this paper, pointed out mistakes, and spend much time
listening to the new techniques developed here.

I want also to thank Peter Haissinski for his careful reading of the first version of this article.

This manuscript was written about 14 years ago around year 2000.

\section {Preliminaries.} \label{prelim}

\subsection {Potential theory.} \label{potential}

For basic background on potential theory we refer
to \cite{[Ra]} and \cite{[Tsu]}. We give a more geometric point 
of view than in the classical expositions, that are
more ``function theoretic'' oriented.

We consider a 
compact and full set $K\subset {\overline {\CC}}$ 
containing at least two points and 
$\infty \notin K$. We denote $\O ={\overline {\CC}}-K$
the complementary domain.
 We assume that $K$ has positive  
capacity and, more precisely, that all points of $K$ are regular 
points for the Dirichlet problem. This is the case 
for a uniformly perfect set, in particular for Julia 
sets of rational functions (\cite{[Ca-Ga]} p.64).

Let $z\mapsto g(z, \infty)=G(z)$ be the 
Green function in $\O$ (the potential of $K$.)  
The potential $G(z)$ converges to $0$ 
when $z\to K$ and $G_K(z)=\log |z| +u(z)$ where
$u$ is a harmonic function, and 
$$
\lim_{z\to \infty} G_K(z)-\log |z| = \g,
$$
where $\g$ is Robin's constant of $\O$. The
capacity of $K$ is $e^{-\g}= \Capa (K)$.
These properties and the fact that $G: \CC-K \to \RR_+^*$
is harmonic and positive determine $G$ uniquely. 
We denote $G=G_K$.

\begin{definition}
An equipotential of $K$ is a  level sets of $G_K$.
\end{definition}

The following Proposition is elementary:

\begin{proposition}
An equipotential is a finite union of analytic 
Jordan curves if and only if it does not contain 
critical points of $G_K$. Each component of an equipotential
is endowed with a natural parametrization (or measure) 
generated by integration of the conjugate function $G_K^*$
(the harmonic parametrization or measure.)
\end{proposition}

\begin{definition} An external ray or a Green line
is an orthogonal trajectory to equipotentials. 
\end{definition}

An we have another elementary proposition:

\begin{proposition}
External rays are 
analytic arcs if and only if they do not contain 
critical points of $G_K$. Otherwise they split at critical points into several
analytic arcs and they are called \textit{critical external rays}.

External rays are endowed with a natural parametrization 
determined by the potential. The first critical point
in a critical external ray is the critical point with 
higher potential.
The sub-critical part of 
a critical external ray is the closed subset of 
points with potential below the first critical point.
\end{proposition}

\begin{definition} {The skeleton $\G (K)$ is the union
of sub-critical parts of critical external rays.}
\end{definition}

The skeleton of a connected set $K$ is empty.
The skeleton of $K$ has always $0$ Lebesgue measure since it is the union of at most a countable 
number of segments, . 
The union $K\cup \G (K)$ is a full connected set.

Let $G_K^*$ be a conjugate function of $G_K$ in a 
neighborhood of $\infty$. The conjugate function 
is determined up to an additive constant. We choose
the additive constant so that 
$\varphi_K (z) =e^{G_K(z)+iG_K^* (z)}$
defines a univalent map angularly tangent to the identity 
in a neighborhood of $\infty$ (this means that the derivative
of $\varphi_K$ at the fixed point $\infty$ is real and positive,
equal to $1/\Capa (K)$.) 
Note that a preimage 
of a circle of large radius centered at $0$ is an 
{\it equipotential }
of $K$, i.e. a constant locus for $G_K$. 
The holomorphic germ $\varphi_K$ defines
by analytic continuation a multivalued function. The 
monodromy around the critical points of $G_K$ introduces
an indeterminacy in the argument. The analytic continuation
of $\varphi_K$ is well defined outside the skeleton.
At points of the skeleton we obtain different limits 
depending where we approach the point from one side or 
the other. The image skeleton 
$\G' (K)\subset {\overline {\CC}}-{\overline {\DD}}$
is the set of all these limits. The image skeleton
is composed by at most a countable number of radial 
semi-open segments having the open end-point on the unit 
circle.

\begin{definition}{The Green map of $K$ is 
the holomorphic diffeomorphism 
$$
\varphi_K : {\overline {\CC}}-(K\cup \G (K)) \to 
{\overline {\CC}}-({\overline {\DD}}\cup \G' (K)) .
$$
}
\end{definition}

When $K$ is connected, the map $\varphi_K$  
is a conformal representation.
Observe that the
function $\varphi_K$ is onto because 
$G_K (z) \to 0$ when $z\to K$ and $|\varphi_K |=e^{G_K}$.
Observe that we can transport by 
$\varphi_K$ conformal
structures on ${\overline {\CC}}-{\overline {\DD}}$ defined almost everywhere.

When $K$ is a Cantor set, 
$K\cup \G (K)$ is connected and locally connected
(in the region of potential larger than $\eps >0$,
$\G (K)$ consists of a finite number of analytic arcs).
Thus when $K$ is a Cantor set, the 
univalent map $\varphi_K$ has a radial extension 
at all points, i.e all non-critical 
external rays do land at some point.

Sometimes it is more natural to work in $\HH /\ZZ$ instead
of $\CC-{\overline {\DD}}$. We denote 
$E: \HH /\ZZ \to \CC -{\overline {\DD}}$ defined by 
$$
E(z) =e^{-2\pi i z}.
$$
The image $E^{-1} (\G' (K))$ is at most a countable 
union of vertical segments in $\HH /\ZZ$. These 
vertical segments
can be grouped in groups with the same length so 
that their tip corresponds
to the same critical point of the Green function.
For each group we can cut these vertical segments and we 
repaste the boundaries respecting the imaginary coordinate.
We do the pasting according to the identifications
of the corresponding parts of the skeleton $\G (K)$ so 
that the Riemann surface $\cS_K$ obtained in that way is
biholomorphic to $\CC-K$. Observe that the complex 
structure in that Riemann surface is generated by a 
flat metric inherited from the flat metric of $\HH/\ZZ$
because we did a geodesic gluing. The surface 
$\cS_K$ endowed with this metric structures is the 
{\it cylindrical model} of $\CC -K$.

The capacity (and the Lebesgue measure as well) 
is an upper-semi-continuous function in the 
space of compact subsets of the plane endowed with 
the Hausdorff topology (see \cite{[Tsu]} p.57).

\begin{proposition}{Let $(K_n)_{n\geq 0}$ be a 
sequence of compact subsets in $\CC$ converging 
in Hausdorff topology to a compact set $K\subset \CC$.
We have
$$
\limsup_{n \to +\infty } \Capa (K_n) \leq \Capa (K).
$$
}
\end{proposition}

Given a compact set $K\subset \CC$ with positive 
capacity and $\eps >0$ we denote by 
$$
W_\eps (K) =\{ z\in \CC ; G_K (z) \leq \eps \}
$$
the $\eps$-potential closed neighborhood of $K$.
Observe that in Hausdorff topology
$$
\lim_{\eps \to 0} W_\eps (K) =K^*,
$$
where $K^*\subset K$ is the {\it carrier} of $K$,
i.e. the subset of regular points, or
the support of the 
equilibrium measure of $K$ (see [Ra] Thm. 4.2.4 p.93,
Thm. 4.3.14 p.105.)

\begin{lemma}{Let $K_1$ and $K_2$ be two compact
sets in the plane with positive capacity, 
and $\eps >0$ such that 
$$
K_1 \subset W_\eps (K_2).
$$
Then we have on $\CC -  W_\eps (K_2)$, 
$$
G_{K_2} \leq G_{K_1} +\eps.
$$
}
\end{lemma}

\begin{proof}{On the boundary of $W_\eps (K_2)$ we have
$$
G_{K_2} -\eps \leq 0 \leq G_{K_1},
$$
thus the maximum principle on the domain 
${\overline {\CC}}-W_\eps (K_2)$ applied to 
$G_{K_2} -G_{K_1} -\eps$ gives the result.}
\end{proof}

Consider now a sequence of compact sets $(K_n)_{n\geq 0}$
in the plane  with all their points regular, 
$$
\Capa (K_n) \geq \eps_0 >0,
$$
and converging in Hausdorff topology to 
a compact set $K$. Then we have $\Capa (K) \geq \eps_0$.
Not all points of $K$ are necessarily regular, 
in general $K^* \not= K$.
We have $\Capa (K^*) =\Capa (K)
\geq \eps_0$ (\cite{[Tsu]} p.56.)
By Harnack's principle (\cite{[Ra]} p.16) we can extract 
a converging subsequence of the sequence of positive 
harmonic functions $(G_{K_n})_{n\geq 0}$ converging 
uniformly on compact subsets of $\CC -K$ 
to a positive harmonic function $G$. It follows from 
the lemma that $G=G_K =G_{K^*}$.
We conclude:

\begin{proposition}{Let $(K_n)_{n\geq 0}$ be a 
sequence of compact sets with all points regular and
$$
\Capa (K_n) \geq \eps_0.
$$
We assume that 
$$
\lim_{n\to +\infty } K_n =K.
$$
Then uniformly on compact subsets of $\CC-K$,
$$
\lim_{n\to +\infty} G_{K_n} =G_K.
$$
Thus
$$
\lim_{n\to +\infty} \varphi_{K_n} =\varphi_{K^*},
$$
in the Caratheodory kernel of the domain of definition
of these Green mappings.
}
\end{proposition}

\subsection{The analytic tree.} \label{tree}

We consider in this section a disconnected compact set 
$K\subset \CC$ with positive capacity and with all 
its points regular.
The complement $\O ={\overline {\CC}} -K$ is 
a planar Riemann surface. As described in section I.1.a
there exists a Green function $G_K$ with logarithmic singularity
at $\infty$ which tends to $0$ in $K$. 
The Riemann surface $\O$ is hyperbolic, 
i.e. Green functions exist.

We study the conformal invariants of $\O$.
Let ${\cC} (K)$ be the set whose elements are 
critical equipotentials, i.e. those containing 
critical points of $G_K$. In the generic situation,
all critical equipotential contain only one critical 
point with multiplicity one (this is not the case
for the Julia set of a generic polynomial of degree
larger than $3$). We say that $K$ is 
generic if this happens. It can be  proved
that this holds for generic $K$ in the Baire sense for 
the Hausdorff metric (Sketch of the proof: The capacity function 
is upper semi-continuous, thus the space of compact 
sets of $\CC$ with capacity larger than $r>0$ is 
a closed set, thus a Baire space. For $N>0$ the set 
of those $K$ whose first $N$ critical equipotentials
contain only one critical point with multiplicity one 
is open and dense). 
It is sufficient for our purposes 
to study the  generic case. The modifications to treat 
the general case are straightforward, but the 
combinatorics is cumbersome. We will sketch the few 
modifications necessary to treat the general case and will 
consider only the generic situation.

The connected components of ${\overline {\CC}} 
-(\{ \infty \} \cup 
\cC (K)\cup K)$ is a 
family of annuli $(A_i)$. We have an order relation 
$A_2 < A_1$ if $A_2$ is contained in the bounded 
component of the complement of $A_1$. For this order 
relation, the unique annulus $A_0$ in the family such that 
$\infty \in \partial A_0$ is a maximal element.
The family $(A_i)$ with the order $<$ is a tree.

\begin{definition}\textbf {(Combinatorial tree).}{ 
The combinatorial tree associated to $K$ or $\O$ is 
$$
\cT_c (K) =\cT_c (\O ) =( (A), <).
$$
}
\end{definition}

In the generic case, this tree is always a dyadic tree,
i.e. the valence of each vertex is $2$. In the general 
case the combinatorial tree has finite valence at each 
vertex because the harmonic function 
$G_K$ has only isolated critical points
and each critical equipotential contains only a finite 
number of critical points.

The combinatorial tree is obviously a conformal invariant 
of $\O$ but it is not the only one. We label each  
vertex $A$ of $\cT_c (K)$ with the modulus  
$\mod A \in ]0, +\infty ]$ of the corresponding annulus. 
 These are the {\it modular invariants}. 

Next we have combinatorial invariants that describe
how to paste the annulus of the tree. The combinatorics
is much simpler in the generic situation : It only 
depends on two angles (and in the symmetric situation, 
that will be specially important for the applications to 
the quadratic Julia sets, 
it depends only on one angle). We treat the generic 
situation. Each 
annulus of the tree that is not the root ($A \not= A_0$) 
or an end 
contains one critical point
in its outer boundary component and one critical point in 
its inner boundary component that 
corresponds to two points of its prime-end circle.
Given such an annulus $A$ that is not 
the root ($A \not= A_0$),
we map conformally $A$ into a round annulus 
(i.e. bordered by circles) centered at the origin 
of the complex plane. We map the outer (resp. inner) boundary
component into the outer (resp. inner) boundary component. 
We denote $z_+$ the image of 
the critical point in the outer boundary component and 
$(z_-^{(1)} , z_-^{(2)} )$ the images of the two access to 
the critical point in the inner boundary component, 
with $z_-^{(1)}$ being the first one encountered when going 
around the annulus counterclockwise from $z_+$.
The pasting invariant is  the unordered pair of 
angles defined modulo $1$
(we measure angles modulo $1$)
\begin{align*}
\t (A) &= \{ \t_1 (A) , \t_2 (A) \},\\
\t_1 (A) &=\frac{1} {2 \pi } (\arg z_+ -\arg z_-^{(1)}) ,\\
\t_2 (A) &=\frac{1} {2 \pi } (\arg z_+ -\arg z_-^{(2)}) .\\
\end{align*}

For the root $A_0$ we proceed as follows. We identify 
$A_0$ to the pointed unit disk $\DD =\DD-\{ 0\}$ and
we consider the two points $(z_-^{(1)} , z_-^{(2)} )$ 
corresponding to the two access of the critical on the 
inner boundary of $A_0$. We define modulo $1$
$$
\t_1 (A_0 )=1-\t_2 (A_0 )=\frac{1}{ 2 \pi } ( \arg z_-^{(1)}
-\arg z_-^{(2)} ).
$$
The angular invariant of the root $\t (A_0)= \{\t_1 (A_0) ,
\t_2 (A_0)\}$ is only  significative modulo the elements
the diagonal, i.e. of the form $(\a , \a)$ 
(only the difference of 
the arguments of $z_-^{(1)}$ and $z_-^{(2)}$ matters).
For an end, we define $\t (A) =0$ (there is no significative
angular invariant for an end).

\begin{definition}\textbf{ (Analytic tree).} {For a generic
$K$ we define the analytic tree of $K$ or $\O$ to 
be the combinatorial tree with labeled vertices
$$
\cT_\o (K) =\cT_\o (\O ) =( (A, \mod (A) , \t (A) , <).
$$

In general we call an analytic tree such a combinatorial
object.
}
\end{definition}

\begin{remark}

The angular invariants in the non-generic situation describe
the relative position of critical points in the outer boundary 
and the position of critical accesses corresponding to 
critical points in the inner boundary. As in the generic
situation all angles are only defined up to one rotation. 
More important, we have to add the information describing 
which sets of accesses do correspond to the same critical 
point.
\end{remark}

\begin{remark}
Continuing this work, N. Emerson developped in his Thesis at UCLA the combinatorics and structure of general Analytic Trees
and their combinatorics (see \cite{[Em]}). 
\end{remark}

\medskip

\textbf {Riemann surface associated to an analytic tree.}

\medskip

From the analytic tree of $\O$ we can reconstruct 
$\O$ by gluing together round annulus (that is, 
bounded by two concentric circles) of modulus 
$\mod (A)$ in the prescribed way by the tree and the 
angular invariants. More precisely, we describe the 
gluing construction that associates to each analytic 
tree $\cT_\o $ a Riemann surface $\cS (\cT_\o )$.

We start with $A_0 \cup \{ \infty \} $ that 
we identify to the exterior of the closed unit disk
${\overline {\CC}}-{\overline {\DD}}$. 
We determine two 
points on the unit circle compatible with the angular
invariant $\t (A_0)$, i.e. two points whose difference 
of arguments modulo $1$ is equal to $\t_1 (A_0) -\t_2 (A_0)$.  
We identify these two  points of the  
boundary. Now there is a unique way  
of gluing a children of $A_0$ so that the harmonic parameterization 
of the identified boundaries do correspond.  If not an end,
the angular invariant 
of the children determines exactly the location on the 
inner boundary of two points. We identify these two 
points and we are again in the same situation for gluing 
the next two children. We go on and paste all children in the 
tree. Finally we paste the ends in such a way that the 
harmonic measures correspond on both sides of 
the pasted boundary.
We denote  $\cS (\cT_\o )$ the abstract Riemann 
surface constructed from the analytic tree $\cT_\o$ in 
that way. Obviously there is a harmonic function on 
$\cS (\cT_\o )$ that generates the analytic tree for 
which the traces of the 
pasted boundaries are equipotentials. Moreover we 
have that $ \cS (\cT_\o (\O ))$ and $\O $ are 
biholomorphic
$$
\cS (\cT_\o (\O )) \approx \O ,
$$
because a conformal representation of $A_0\cup 
\{ \infty \}$ into its
trace in $\cS (\cT_\o (\O ))$ extends into a 
holomorphic diffeomorphism from $\O$ into  
$\cS (\cT_\o (\O ))$ because the pastings are 
compatible (in a neighborhood of a 
non-singular glued point the harmonic functions 
on both sides have conjugate functions that do 
coincide on the boundary, so the harmonic functions and
its conjugates do define analytic extensions of the same
local holomorphic diffeomorphism).
Observe that the pointed Riemann surfaces with 
the point $\infty$ for $\O$ and the corresponding point to 
$\infty$ for $\cS (\cT_\o (\O ))$ are biholomorphic also.
From this construction we obtain:

\begin{theorem} { The pointed planar Riemann surfaces
$(\O_1 , \infty )$, 
$\O_1 ={\overline {\CC}}-K_1$, and  $(\O_2 , \infty )$,
 $\O_2 ={\overline {\CC}}-K_2$, are biholomorphic
(i.e. by a holomorphic diffeomorphism fixing $\infty$) 
if and only if their analytic trees coincide
$$
\cT_\o (\O_1 ) = \cT_\o (\O_2 ).
$$
}
\end{theorem}

\begin{proof}{Obviously by construction 
$\cT_{\o} (\O)$ only depends on the conformal type 
of $(\O , \infty )$. Conversely, if the analytic trees agree
then 
$$
\O_1  \approx \cS (\cT_\o (\O_1 ))\approx 
 \cS (\cT_\o (\O_2 ))\approx \O_2  \ ,
$$
are all biholomorphic as pointed Riemann surfaces.
}
\end{proof}

We can observe that there are analytic trees that 
can not be realized as the complement of a 
compact set with positive capacity. Just consider
an analytic tree with asymptotic very large 
modulus invariants. The abstract pasted Riemann 
surface is then parabolic, i.e there are no Green 
functions, (use the criterium in \cite{[Ah-Sa]}
IV.15B p.229) and is not the complement of a compact set with 
positive capacity. Note that we could  define 
analytic trees also in general for the complement
of a compact set of zero capacity (or of compact sets 
with irregular points) using an Evans
potential (see \cite{[Sa-Na]} V.3.13A p. 351 or \cite{[Tsu]} III.6 p.75). 
Such a tree 
depends then on the choice of the Evans potential which
is not unique.

On the other hand we can observe that any tree corresponds
to an analytic tree of the complement of a compact set
(in the case of zero capacity the tree is constructed from
the choice of an Evans-Selberg potential). The abstract Riemann 
surface $\cS (\cT_\o )$ 
obtained from the pasting of abstract annulus according
to the combinatorics of the tree $\cT_\o$ is planar, i.e.
every cycle on $\cS (\cT_\o )$ is dividing (see [Ah-Sa] 
I.30D p.66). Then $\cS (\cT_\o )$ can be holomorphically 
embedded in $\CC$ (\cite{[Ah-Sa]} III.11A p.176), 
$\cS (\cT_\o ) \approx \O \subset 
{\overline {\CC}}$, with 
$\infty$ corresponding to the non pasted end 
of $A_0$.

Observe that we can read on $\cT_\o (K)$ many properties
of $K$ or $\O$. The Riemann surface $\O$ is 
finitely connected if and only if $\cT_\o (\O )$ is 
finite. An infinite branch of $\cT_\o (K)$ corresponds
to a point in the ideal boundary of $\O$ and to a 
connected component of $K$.  
If the sum of the modulus invariants are infinite 
along this branch with the root removed, 
i.e. for $(A_1, A_2, \ldots)$,
$$
\sum_{i\geq 1} \mod \ (A_i ) =+\infty ,
$$
then this component is a point.
We observe that if $\cT_\o (K)$ has no finite 
branches and this condition holds for all infinite
branches then $K$ is a Cantor set. Observe also 
that this is not a necessary condition. But if we 
have a finite branch then $K$ contains a non-trivial 
connected component, that is not a point because otherwise
this isolated point will not be regular (contradicting 
one of our assumptions).

We study in more detail the Cantor set situation.
More precisely we want to determine when the analytic
tree (or equivalently the conformal type of 
$\O$) determines $K\subset {\overline {\CC}}$ up to 
a Mo\"ebius transformation. Observe that this problem 
is hopeless when $K$ has a non-trivial connected 
component (because of Riemann uniformization theorem), 
or when $K$ has positive measure (because of 
Morrey-Ahlfors-Bers theorem), or for a "flexible" 
Cantor set as those constructed by R. Kaufman and C. Bishop
\cite{[Bi]}. 
In order to address this 
problem we review in the next section the classical 
theory of $O_{AD}$ Riemann surfaces.

We observe also that all the information of the analytic
tree can be recovered from $\varphi_K$ and the 
skeleton $\G' (K)$. In particular, the preimages of 
annulus of the analytic tree by $\varphi_K$ are 
quadrilaterals bounded by circles and radial segments.
The modulus of the rectangles are the modulus of the 
corresponding annulus. 

The cylindrical model $\cS$
is very convenient. The flat metric on the cylindrical 
model is the extremal metric for all annulus in the 
analytic tree (we have a simultaneous uniformization). 
Observe that if an annulus $A\subset \cS$ 
corresponds to  an annulus in 
the analytic tree then 
$$
\mod A =\frac{1}{ 2 \pi } \frac {g_+ (A ) -g_- (A)} {\mu_H (A)},
$$
where $g_- (A) < g_+ (A)$  are the potentials of the 
boundaries of $A$, and $\mu_H (A)$ is the harmonic measure 
of $A$, i.e. the linear measure of the angles of external 
rays intersecting $A$.

\subsection {The class of $O_{AD}$ Riemann surfaces.}\label{OAD}

 We review part of the classical classification theory 
of Riemann surfaces, and more precisely the results
concerning the $O_{AD}$ class. We remind the classical classification theory 
of Riemann surfaces. As a basic reference the 
reader can consult \cite{[Ah-Sa]}, and for a more encyclopedic 
treatment \cite{[Sa-Na]}.
 
A Riemann surface $\cS$ is in $O_{AD}$
if there are no non-constant AD-functions, i.e. 
holomorphic functions $F$ on $\cS$ with finite
Dirichlet integral
$$
D_{\cS} (F) =\int \int_{\cS} |F'|^2 \ dx dy\ .
$$

By Rado's theorem, any Riemann surface (defined to be connected)
is $\s$-compact.
Consider an exhaustion $(U_n)_{n\geq 0}$ of $\cS$
by open sets $U_n$ relatively compact in $\cS$ and 
${\overline {U_n}}\subset U_{n+1}$. Consider the 
connected components $(U_{nj})_j$  of $U_{n+1}-U_n$.
We denote by $m_{nj}$ the extremal length of paths in 
$U_{nj}$ joining $\partial U_n \cap {\overline {U_{nj}}}$
to $\partial U_{n+1} \cap {\overline {U_{nj}}}$.
We have ([Ah-Sa] IV.13D p.225) that 
$$
m_{nj} =\frac{1}{ D_{U_{nj}} (u) }\ ,
$$
where 
$$
D_{U_{nj}} (u) =\int \int_{U_{nj}} 
|\grad u |^2 dx \ dy ,
$$
where $u$ is the harmonic function in $U_{nj}$ equal 
to $0$ in $\partial U_n \cap {\overline {U_{nj}}}$ 
and equal to $1$ in 
$\partial U_{n+1} \cap {\overline {U_{nj}}}$.

Define 
$$
m_n =\inf_j m_{nj}.
$$

Then we have the following modular test due to L. Sario 
(\cite{[Ah-Sa]} IV.16D p.233):

\begin{theorem} {
If there exists an exhaustion
$(U_n)_{n\geq 0}$ such that 
$$
\sum_{n=0}^{+\infty} m_n =+\infty 
$$
then $\cS$ belongs to the class $O_{AD}$.}
\end{theorem}

Observe that the strength of Sario's criterium relies
on the choice of the exhaustion. Given a Riemann surface
it is always possible to choose bad exhaustions for
which the criterium fails. The most efficient choice of
the exhaustion consists in those that give modulus $m_{nj}$
that are comparable in magnitude for fixed $n$ (the boundaries
of the exhausting domains cut "evenly" the Riemann surface).
Observe that Sario's criterium gives a sufficient condition
only. There is a distinct class of Riemann surfaces for 
which Sario's criterium holds that is strictly distinct 
from the $O_{AD}$ class. For these Riemann surfaces the $O_{AD}$
property can be checked using Sario's geometric critetrium.
Thus it is natural to denote this class by $GO_{AD}$ for 
{\it geometric} $O_{AD}$ class.
Geometric  $O_{AD}$ Riemann surfaces enter in  
the effective explicit application of the 
rectification theorems presented in this article.

\medskip

\textbf {Planar $O_{AD}$ Riemann surfaces.}

\medskip

We consider $\cS ={\overline {\CC}}-K$ where 
$K$ is a compact set in $\CC$.
The following results can be traced back to L. Ahlfors,
A. Beurling, R. Nevanlinna and L. Sario 
(\cite{[Ah-Sa]} IV.2B p.199).

\begin{theorem}\label{thm_OAD}{
The following conditions are
equivalent :

\item {(i)} The Riemann surface $\cS =
{\overline {\CC}}-K$ is in the class $O_{AD}$.

\item {(ii)} The compact set $K$ has absolute
measure zero: For any univalent map defined on the 
complement of $K$, the complement of the image has 
measure zero.

\item {(iii)} The compact set $K\subset \CC$ is 
rigidly embedded in the complex plane: The only 
univalent functions on the complement of $K$ are 
Moebius transformations.

\item {(iv)} The compact set $K\subset \CC$ is 
neglectable for extremal length: The modulus of a 
path family $\G$ is the same than the modulus of the 
path family whose ``paths'' are the ones of $\G$ removing 
their intersection with $K$.

}
\end{theorem}

What makes the class $O_{AD}$ particularly interesting
for our purposes is the strong removability property (iii).
The complement of a Cantor quadratic (or cubic, see \cite{[BH]})
Julia set is in $O_{AD}$. C. McMullen seems to have been 
the first to notice that. He formulated a 
useful variant of 
Sario's test for planar Riemann surfaces 
using modulus of annulus (see \cite{[Mc]} p.20):

\begin{theorem} { 
We assume that we have a tree whose vertex are annulus
$A\subset \CC-K$  such that if 
$B$ is a children of $A$ then $B$ is contained
in the bounded component $I(A)$ of $\CC -A$. If 
for any $z\in K$ there is an infinite branch $(A_i (z))$
deprived from the root 
with $z\in I (A_i (z))$ for all index $i$ and 
$$
\sum_{i\geq 1} \mod A_i (z) =+\infty \ ,
$$
then the Riemann surface $\cS$ is in the class $O_{AD}$.}
\end{theorem}

\begin{remark} 

In the application to quadratic renormalization 
we use this geometric criterium in a 
weak form : All modulus of annuli will be bounded
from below.

\end{remark}

\bigskip

\begin{definition}\textbf {(Thin tree).}{A modular tree is a tree 
of annulus nested inside each other in the prescribed way by 
the tree, and vertices labeled by the modulus of 
the corresponding annulus. Analytic trees have an 
underlying modular tree structure. 

A modular tree is thin if it has only infinite branches
and along all branches the sum of the modulus of the 
annuli except the root is infinite.}
\end{definition}

McMullen's test can be applied to an Analytic 
Tree:

\begin{theorem}\label{thm_analytic_tree}{Let $\cT_\o$ be an analytic 
tree which is thin. 
Then the corresponding planar Riemann surface 
$\cS (\cT_\o ) \equiv 
\CC -K (\cT_\o )$ is in 
the class $O_{AD}$, i.e. $K(\cT_\o )$ is a Cantor
set of absolute measure zero and its embedding in the
plane is uniquely determined up to a Moebius transformation.}
\end{theorem}

Notice that the 
condition can be checked in a purely combinatorial way.

One can prove the previous 
theorem using Sario's test choosing a suitable 
exhaustion $(U_n)$ of $\cS (\cT_\o )$
where the components of the boundary of $U_n$ are 
composed by pieces of equipotentials (with different
potentials) and all modulus $(m_{n_j})_j$ are the same.

\medskip

\begin{definition} \textbf {(Thin and totally thin Cantor sets).}{
A tree of annulus $\cT_\o$ as in theorem \ref{thm_analytic_tree} for a compact set $K$
will be called a tree associated to $K$ if $K=K(\cT_\o )$. In particular 
the analytic tree of $K$ is a tree associated to $K$.

A Cantor set $K$ 
is a thin Cantor set if there is an associated 
tree which is thin. Then $\CC -K$ is an $O_{AD}$ Riemann 
surface and $K$ has absolute measure zero. 

A Cantor set $K$ of positive capacity with all points regular
is a potentially thin Cantor set if its 
analytic tree is thin.
}
\end{definition}

Note that in the definition of thin Cantor set we include
Cantor sets of zero capacity (but not for 
potentially thin Cantor sets).

\subsection {Potential and virtual conformal structures.} \label{conf_struct}

We define in this section {\it potential conformal structures}
and the notions of {\it Green}, {\it equipotential} and 
{\it potential} equivalence between conformal structures.

\begin{definition} \textbf {(Equipotential, Green and potential 
equivalence).} {Let $K \subset \CC$ be a compact set 
with all points regular for the Dirichlet problem. 
A conformal structure $\xi$ 
on $\CC-K$ is 
equipotentially (resp. Green, potentially) 
equivalent to another conformal structure 
$\eta$ on $\CC -K$ if
\begin{align*}
\xi' &=(\varphi_K^{-1} \circ E)^* \ \xi \\
\eta' &=(\varphi_K^{-1} \circ E)^* \ \eta ,\\
\end{align*}
are equi-potentially (resp. Green, potentially) equivalent
 in $\HH /\ZZ$, that
is, there exists 
$$
L : \HH \to \HH
$$
of the form $(x,y) \mapsto (d(x), y)$ (resp. $(x,y)\mapsto 
(x, k(y))$; $(x,y)\mapsto (d(x), k(y))$), 
where $d:\TT \to \TT$ is an absolutely continuous 
homeomorphism (resp. and  $k : \RR_+ \to \RR_+$ is an absolutely
continuous increasing homeomorphism) and 
$$
\eta' =L_* \xi' .
$$
}
\end{definition}

\begin{definition} \textbf {(Potential conformal structure).}{
A potential conformal structure is a conformal structure
potentially equivalent to $\s_0$.
}
\end{definition}

\begin{remark} 

The ellipses in the tangent spaces 
defining a potential conformal structure  
have their principal axes tangent or orthogonal to 
equipotentials. Observe that $L$ is differentiable 
almost everywhere thus a potential conformal structure 
has a well defined Beltrami form.

\end{remark}

\bigskip

A potential conformal structure is locally quasi-conformal 
on ${\overline {\CC}}-K$ if and only if both $d$ and 
$k$ are locally lipchitz.

In other words, a potential conformal structure $\xi$ is compatible with the 
potential of $K$. For the potential theory with respect to 
the  new conformal structure equipotentials
and Green lines are the same than for the potential theory 
with respect to the standard conformal structure.  
More precisely we have the following theorem:

\begin{theorem}{Assume that $\xi$ is a quasi-conformal 
potential conformal structure with the above notations. 
Let $h_\xi :
{\overline {\CC}} \to {\overline {\CC}}$ be a rectification
of $\xi$ such that $h_\xi (\infty ) =\infty$. 
Then all points of $h_\xi (K)$ are regular and 
if $K'=h_\xi (K)$ we can choose a normalization of $h_\xi$
such that the map  
$$
\varphi=h_\xi^{-1} \circ \varphi_K \circ E^{-1}
\circ L \circ E
$$
is tangent to the identity at $\infty$ and then we have
$$
\varphi_{K'}=\varphi .
$$
In particular $h_\xi (\G (K))$ is the skeleton of $K'$
and  for $z\in {\overline {\CC}}-K$,
\begin{align}
G_{K'} (h_\xi (z)) &=k ( G_K (z)) \\
\t_{K'} (h_\xi (z)) &= d (\t_K (z)) .\\
 \end{align}

}
\end{theorem}

\begin{proof}{The proof is straightforward. The 
mapping 
$$
\varphi =h_\xi^{-1} \circ \varphi_K \circ E^{-1}
\circ L \circ E
$$
is holomorphic and defines a holomorphic diffeomorphism
in a neighborhood of $\infty$. With the appropriate
normalization of $h_\xi$ (composing with a linear dilatation)
we can make $\varphi$ tangent to the identity at $\infty$.
Then $\log |\varphi |$ is the Green function of $K'$.}
\end{proof}

We observe that a potential conformal structure is uniquely
determined by the compact $K$ and by the mapping $L$, or 
more precisely by the mappings $d$ and $k$.
Now if $d$ is a uniform limit of orientation homeomorphisms
of the circle $d_n$ we can think to $(d, k)$ as defining 
a singular conformal structure $\xi$ even in the case
when $d$ is not a homeomorphism (for example when it maps intervals into single points). 
The main 
result of next section proves that under natural restrictions
on $d$ and $k$ the rectifications of $\xi_n$ defined by 
$(d_n, k)$ do converge to a "rectification" of $\xi$.

\medskip

\textbf {Virtual conformal structures.} 

\medskip

Consider a metric space $X$.
Let $\cC (X )$ be the space of closed 
continuous correspondences on $X$, this is 
the set of closed connected subsets of $X \times
X$ endowed with the Hausdorff topology associated 
to the product distance in $X \times X$.
Recall that correspondences can be composed: If 
$d_1 , d_2 \in \cC (X )$,
$$
d_2\circ d_1 =\{ (x,z) \in X^2 ; {\hbox {\rm 
there exists }} y\in X, \ 
(x,y)\in d_1 {\hbox {\rm   and  }} (y,z)\in d_2 \} .
$$

Consider now $\Homeo(X)$ the 
space of homeomorphisms 
of $X$ endowed with the topology of uniform
convergence.

The graph map that associates to an element
$\varphi \in \Homeo (X)$ its 
graph
${\cG }:  {\Homeo (X)} \to {\cC ( X)} $,

$$ 
 {\varphi } \mapsto {\cG (\varphi )=
\{ (x,y)\in X^2 ; y=\varphi (x) \}}
$$
gives a continuous embedding of 
$\Homeo (X)$ into $\cC (X)$.

\begin{definition}\textbf{(Map correspondences).}{

We define 
$$
\cC_0 (X )={\overline {
\cG (\Homeo_+ (X) )}}\subset
\cC (X)
$$ 
to be the space of limit 
homeomorphisms, and we define 
$$
\hat {\cC}_0 (X) =\cC_0 (X) \cap C^0 (X,X)
$$
to be the space of map correspondences (where $C^0 (X,X)$ denotes
the spaces of continuous maps from $X$ into $X$).}
\end{definition}

We are interested in the case $X=\TT$.
The following proposition is immediate.

\begin{proposition}{
The space $\cC_0 (\TT)$ is homeomorphic to 
the space $\cM (\TT )$ of probability 
measures on $\TT$ endowed with the weak 
topology. We have a homeomorphism
${\int} : {\cM (\TT )}\to {\cC_0 (\TT )}$
$$
{\mu } \mapsto {d=\int d\mu }
$$
where $d=\int d\mu$ is defined by the 
fact that $d\cap ( \{ x\} \times \TT )$ is always 
a point or an interval and 
\begin{align*}
\left | d\cap ( \{ x\} \times \TT ) \right | &=
\mu (\{ x \} ) ,\\
\sup_{\genfrac{}{}{0pt}{}{(x,y)  \in d}{0\leq x \leq x_0}} y &=
\mu ([0,x_0]) =\int_0^{x_0} \ d\mu .\\
\end{align*}

The same map induces a homeomorphism from the space
$\cM_{na} (\TT)$ of non atomic probability measures
into the space of map correspondences $\hat {\cC }_0 (
\TT)$.
}
\end{proposition}

To each homeomorphism $d\in \Homeo_+ (\TT)$
we have associated before a homeomorphism 
$L:\HH \to \HH$. Now to $d\in \hat {\cC}_0 (\TT)$
and $k\in \Homeo_+ (\RR_+)$
we can also associate in the same way a map correspondence of 
$\HH /\ZZ$, $L\in \hat {\cC}_0 (\HH)$,
${L} : {\HH /\ZZ} {\HH /\ZZ}$,
$$
{(x,y)}  \mapsto {(d (x) , k(y))}.
$$

From a classical point of view, it doesn't make sense to talk about the 
``complex structure'' $L_* \s_0$, but all the information 
is contained in $L$ or more precisely on the pair
$(d,k) \in \hat {\cC}_0 (\TT) \times 
{\Homeo}_+ (\RR_+)$.

\begin{definition} \textbf{(Virtual conformal structure).}
{A virtual conformal structure on $\HH /\ZZ$ or
${\overline {\CC}}$ 
is identified with a map correspondence 
$d\in \hat {\cC}_0 (\TT )$ (or with 
its corresponding non-atomic probability measure 
$\mu_d \in \cM_{na} (\TT )$) and an orientation
preserving homeomorphism $k \in {\rm Homeo}_+ (\RR_+)$.}
\end{definition}

\bigskip

{\bf Observation.}

We can give a more general definition allowing $k\in 
\hat {\cC}_0 (\RR_+ )$ and prove similar 
rectification theorems but this is unnecessary for the 
applications.

\subsection {Rickman's theorem.}\label{rickman}

We use at different places 
Rickman's removability theorem for 
quasi-conformal mappings (see \cite{[Ri]}).
Rickman's theorem has been used before 
in the theory of polynomial-like mappings
(see \cite{[DH]}). Rickman's theorem holds in higher
dimension. 

\begin{theorem} (S. Rickman) \label{thm_rickman}
{Let $U\subset \CC$ be  an open set, 
$K\subset U$ be a closed in $U$ (for the relative 
topology in $U$) , $f$ and $g$ be 
two mappings $U\to \CC$ which are homeomorphisms
onto their images. Suppose that $g$ is 
quasi-conformal, that $f$ is quasi-conformal 
on $U-K$ and that $f =g$ on $K$. Then 
$f$ is quasi-conformal, and $D f =D g$ 
almost everywhere on $K$.}
\end{theorem}

We use in the applications to renormalization 
a version of Rickman's theorem
for quasi-regular mappings. First 
we recall the definition
of a topological branched covering.

\begin{definition} {A map $f$ is a topological 
branched covering if $f$ is locally the composition of 
a holomorphic map and a homeomorphism.

The map $f$ is quasi-regular if $f$  is 
locally the 
composition of 
a holomorphic map and a quasi-conformal homeomorphism.}
\end{definition}

\begin{theorem}{Let $U\subset {\overline {\CC}}$ be 
an open set and $K\subset U$ be a closed set in $U$.
Let $f$ and $g$ be two topological branched coverings
from $U$ into their images.

We assume that $g$ is quasi-regular,  
$f$ is quasi-regular on $U-K$, and that 
$f=g$ on $K$.
Then $f$ is quasi-regular and almost everywhere in $K$
$$
Df =Dg .
$$
}
\end{theorem}

\begin{proof}{At each point $z$ which is not 
a critical point of $f$ or $g$ we use Rickman's 
theorem for homeomorphisms in a neighborhood
$V$ of $z$ to get $Df =Dg$ almost everywhere
in $ K\cap V$ ($K\cap V$ is closed in $V$ for 
the relative topology in $V$). 
The remaining points form a 
discrete set in $U$ of zero Lebesgue measure.}
\end{proof}

\section {Rectification of degenerate conformal structures.} \label{rectification}

\subsection {Introduction.}

Morrey-Ahlfors-Bers theorem proves the existence of 
solutions of the Beltrami equation
$$
\bar \partial f =\mu  \ \partial f
$$
when $\mu$ is measurable and bounded $||\mu ||_{L^\infty} <1 $.
This is a very powerful theorem of rectification 
of conformal structures quasi-conformal with respect 
to the standard structure.
Numerous efforts have been devoted to weak the boundeness 
assumption, i.e. to remove the quasi-conformality assumption.  
O. Lehto \cite{[Le1]} \cite{[Le2]} and G. David \cite{[Da]} proved  
existence theorems for degenerate 
Beltrami forms with $||\mu ||_{L^\infty} =1$ 
with assumptions on the measure of the set where $\mu$ 
degenerates (i.e. where $\mu $ gets close 
to $1$ in absolute value). Typically David's 
condition requires that the measure of the set where 
$\mu$ degenerates decreases exponentially fast:
$$
|\{ z\in \CC ; 1-\eps < |\mu (z) |\}| \leq C_0 e^{-\frac{1}{\eps}}.
$$
Lehto's condition is more general, allows cancellations, 
and is expressed in an integral
form (we refer to Lehto's articles cited before for a precise definition). This type of conditions is necessary 
in order to preserve 
the topology (i.e. to have a rectification that is indeed a homeomorphism).
It is easy to construct conterexamples where the topology 
is destroyed (for example when there are annulus with infinite 
modulus have finite modulus with respect to the new conformal 
structure). 
The solutions that these authors obtain are always almost 
everywhere differentiable homeomorphisms. 
Unfortunately these theorems 
don't seem to be adapted to the type of wild "collapsing"

We present in this section new general theorems of existence 
of solutions of "the Beltrami equation" for  
Beltrami forms $\mu$ such that $||\mu ||_{L^\infty} =1$.

In the first theorem (section \ref{rectif_thm}) the Beltrami form $\mu$ will be 
bounded away from $1$ on compact subsets of  
$\CC-K$ where $K$ is a thin Cantor set. The main 
requirement is that there is a tree associated to $K$
that  is thin with respect to the {\it new} conformal structure.
Then the solution will be a unique 
almost everywhere differentiable homeomorphism.
The same idea will prove the existence of 
solutions of the Beltrami equation
in a unique way for a class of extremely degenerate Beltrami
forms not even defining quasi-conformal structures
in compact subsets of the complement of $K$, which this
time is assumed to be potentially thin (section \ref{rectif_thm}) The
conformal structures defined by such degenerate Beltrami 
coefficients are of a very special nature. They are {\it potential}
conformal structures which in general terms it means that 
they are compatible with the potential of $K$. The solutions 
will still be homeomorphisms.
The next step (section \ref{rectif_virt}) generalizes the previous 
rectification theorems to {\it virtual } conformal structures
that are not conformal structures in a classical sense.
They are defined by sequences of degenerating potential 
conformal structures. These rectifications do have 
a uniquely determined limit that can be though of as the 
"rectification" of the virtual conformal structure.  
In general the solutions this time will not even be homeomorphisms.
They will be continuous mappings from the Riemann sphere 
into itself.
For example, we will show examples where a full measure dense open set is 
collapsed into a set of measure zero.

The idea at the base of these new rectification theorems is to 
consider Beltrami coefficients that do degenerate into 
a thin or potentially thin Cantor set. The removability of this thin
Cantor set is used to construct the unique rectification.
More precisely the image of the thin Cantor set is uniquely 
determined and is the Hausdorff limit of its image by rectifications
of approximating quasi-conformal structures. When considering
conformal structures that do respect the potential theory 
in the exterior of the potentially thin Cantor set, from the 
convergence of the images of the Cantor sets we get the 
convergence of the potential and therefore the convergence of the 
rectifications to a unique limit mapping.

\subsection { First rectification theorems.}\label{rectif_thm}

\begin{theorem}\label{rectif1}{We consider a compact set $K\subset \CC$.
Let $\xi$ be a conformal structure on $\CC-K$. We assume that 
$\xi$ is quasi-conformal on compact subsets of $\CC-K$.
We denote by $\mu$ its Beltrami coefficient.
The locally rectifiable conformal structure $\xi$ defines a new
Riemann surface structure on $\CC-K$. We assume that 
$(\CC-K , \xi )$ is a Riemann surface in the class $O_{AD}$.

There exists a continuous mapping $h: ({\overline {\CC}}, \xi )
\to ({\overline {\CC}}, \s_0)$ rectifying the conformal 
structure $\xi$, i.e. the restriction $h_{/\CC-K}$ is a 
homeomorphism that is quasi-conformal on compact subsets of 
$\CC-K$ and 
is a solution to the Beltrami equation
$$
\bar \partial h =\mu \ \partial h .
$$
Moreover $h$ is unique up to composition to the left by a 
Moebius transformation. Thus if $z_0$ and $z_1$ are 
two points in $\CC-K$, $h$ is uniquely determined 
by the normalization $h(z_0)=z_0$, $h(z_1)=z-1$ and $h(\infty ) =\infty$. 
If $K$ is totally disconnected then 
$h$ is a homeomorphism of the Riemann sphere. Otherwise $h$ 
collapses components of $K$ into points. 

Also if we truncate $\mu$ in an $\eps_n$-neighborhood
of $K\cup \{\infty \}$, $\mu_n =\mu . {\hbox {\bf 1}}_{\CC- V_{\eps_n}
(K\cup \{ \infty \})}$, and we consider the classical 
rectifications $(h_n)$ of $(\mu_n) $ normalized as $h$,
$h_n(z_0)=z_0$, $h_n(z_1)=z_1$ and $h_n(\infty ) =\infty$. 
Then $h_n \to h$
uniformly on the Riemann sphere.
}
\end{theorem}

\bigskip

\begin{proof}

\medskip

Consider the Riemann surface $\cS =\CC-K$ endowed 
with the new conformal structure $\xi$. Local rectifications
of $\xi$ provide charts for the new complex structure on $\cS$.
Obviously for this new complex structure $\cS$ is still a 
planar Riemann surface (it is a topological property). 
Thus there is an embedding of 
 $\cS$ in $\CC$ (with 
the end corresponding to $\infty$ corresponding to $\infty$).
We obtain in this way a rectification mapping $h: \CC-K 
\to \CC -K_\infty$. The compact set $K_\infty$ is a totally
disconnected set 
set of absolute area zero.  So $h$ extends continuously 
into a map $h :{\overline {\CC}}\to {\overline {\CC}}$
mapping each connected component of $K$ into the corresponding
point of $K_\infty$. 

We prove the uniqueness. Let $h$ and $h'$ 
be two solutions satisfying the theorem. Then 
$h'\circ h^{-1} : \CC-h (K) \to \CC -h'(K)$ is a 
holomorphic diffeomorphism. Since $\CC-h(K)$ is 
a $O_{AD}$ Riemann surface, using theorem \ref{thm_OAD} 
condition (iii) we have that $h'\circ h^{-1}$
is a Mo\"ebius transformation as claimed.

It remains to prove the second part of the theorem about the 
approximation of $h$ by classical rectifications. We obtain 
also another proof of the existence of $h$.
Assume that $z_0$ and $z_1$ are two distinct points in 
$\CC-K$. Consider a sequence $(\eps_n)$, $\eps_n >0$, 
$\eps_n \to 0$, and consider the $\eps_n$-neighborhood 
(for the chordal metric) of $K\cup \{ \infty \}$, 
$V_{\eps_n} (K\cup \{ \infty \})$. We truncate $\mu$
near $K\cup \infty$. Let 
$\mu_n =\mu \ . {\hbox {\bf 1}}_{\CC-V_{\eps_n}(K\cup \{ \infty \})}$. 
The Beltrami form 
$\mu_n$ satisfies $||\mu_n ||_{L^\infty } <1$ and 
defines a classical quasi-conformal structure. 
Let $\xi_n$ be its 
associated complex structure.
Consider a sequence of classical Morrey-Ahlfors-Bers 
rectifications $h_n : ({\overline {\CC}} ,\xi_n) 
\to ({\overline {\CC}} , \s_0 )$ normalized such that 
$h_n (z_0)=z_0$, $h_n (z_1)=z_1$ and $h_n (\infty )=\infty$.
We prove that the sequence $(h_n)$ converges uniformly to the 
desired $h$.

Observe that the sequence $(h_n)$ is equicontinuous in  
$\CC-K$ because the $h_n$ are uniformly H\"older on compact 
subsets of $\CC-K$ because they are quasi-conformal (see \cite{[Le-Vi]} for classical results on quasi-conformal theory). 
Thus extracting a subsequence we can 
also assume that
$$
\lim_{k\to +\infty} h_{n_k} =h
$$
uniformly on compact subsets of $\CC - K$. 
Observe that by  the same classical proof in 
quasi-conformal theory (see \cite{[Le-Vi]}), 
the limit 
$h$ is a homeomorphism from $\CC-K$ into its image
and is a 
quasi-conformal homeomorphism on compact subsets of 
$\CC-K$. Moreover $h (\CC-K) =\CC -K_\infty$ is 
biholomorphic to $(\CC-K , \xi)$, thus it is a
Riemann surface in the class $O_{AD}$ and $K_\infty$ 
is totally disconnected of absolute area $0$. Thus 
the limit $h$ is unique and we have on $\CC-K$ 
$$
h=\lim_{n\to +\infty} h_n.
$$
Also if $C$ is a component of $K$ then $h_n (C)$ converges
to a point of $K_\infty$ (the point corresponding to 
the corresponding end of $\CC-K$.) Thus on $K$ the 
mappings $h_n$ converge pointwise to an extension of $h$
still denoted by $h :{\overline {\CC}}\to {\overline {\CC}}$.
The convergence is uniform on the Riemann sphere. Consider 
an exhaustion $(U_i)$ of $\CC-K$ with ${\overline {U_i}}$
a compact set in $\CC-K$. Let $\eps >0$ and choose $i$ large
enough so that the diameter of the components of the complement 
of $h({\overline {U_i}})$ have diameter less than $\eps/2$.
The sequence  $(h_n)$ converges uniformly on $U_i$ to $h$.
There is $N\geq 1$ so that for $n\geq N$, $||h-h_n||_{C^0 (U_i)}
\leq \eps /2$. Thus, for $n\geq N$,  the diameter of the image by $h_n$ of 
a component of the complement of ${\overline {U_i}}$ is less 
than $\eps$ and $||h-h_n||_{C^0 ({\overline {\CC}})}
\leq \eps $. Q.E.D.
\end{proof}

\subsubsection{Rectification theorem for thin Cantor sets.} \label{rectif_thin}

In the applications we need an effective version of this 
theorem. First we have to be able to check that the 
Riemann surface $(\CC-K , \xi )$ is in the class $O_{AD}$.
For this we can use Sario's criterium or McMullen's version of 
it. In practice $K$ is not an arbitrary compact set in the 
plane but a thin Cantor set with an associated tree of annulus.

Let $\xi$ be a conformal structure on $\O={\overline {\CC}}-K$
such that $\xi$ is quasi-conformal on compact subsets of 
$\CC-K$ (observe that we allow $\xi$ to degenerate in a 
neighborhood of $\infty$). Observe that we can consider
modulus of annulus relatively compact in 
$\CC-K$ with respect to $\xi$: The modulus
$$
\mod_\xi A
$$
is the modulus of $h(A)$ where $h$ is a quasi-conformal 
rectification of $\xi$ in a neighborhood of $A$. Observe
that this quantity is independent of the rectification chosen.
We denote $\mu$ be the Beltrami form associated to $\xi$
(defined to be $0$ on $K$ and at the point $\infty$, i.e. we consider 
there $\xi =\s_0$).

We assume the following fundamental hypothesis (we 
say that $\xi$ is {\it admissible}) :

\bigskip

{\bf (H)} {\it There exists a thin tree $\cT$ associated to $K$ 
such that $\cT$ is also thin with respect to the 
conformal structure $\xi$, i.e. for any infinite 
branch $(A_i)$ of $\cT$ deprived from the root we have 
$$
\sum_i \mod_\xi A_i =+\infty ,
$$
and 
$$
\mod_\xi A_0 =+\infty.
$$
}

\bigskip

Since $A_0$ is a pointed disk at $\infty$ and so it is not
relatively compact in $\CC-K$, $\mod_\xi A_0 =+\infty$ requires
an explanation. It means that 
$$
\lim_{\eps \to 0} \mod_\xi A_0-V_\eps (\infty ) =+\infty ,
$$
where $V_\eps (\infty )$ denotes an $\eps$ neighborhood of 
$\infty$ for the chordal metric.

Note that by McMullen's criterium the condition $(H)$ implies 
that $(\CC-K, \xi)$ is a Riemann surface in the class 
$O_{AD}$. Moreover since $K$ is a Cantor set, the rectification 
given by the previous theorem is a global homeomorphism of 
the Riemann sphere. Also, transporting the tree of annulus 
by $h$ and using McMullen's criterium, we have that $h(K)$ is a 
thin Cantor set. In conclusion we have:

\begin{theorem}\label{rectif2}{Let $\xi$ be an admissible conformal 
structure on $\CC-K$. We denote by $\mu$ its Beltrami 
form.

There exists a unique absolutely continuous 
almost everywhere differentiable 
homeomorphism rectifying the conformal structure $\xi$,
$h : ({\overline {\CC}}, \xi ) \to ({\overline {\CC}},
\s_0)$, i.e. solving the Beltrami equation
$$
\bar \partial h =\mu \ \partial h,
$$
with $h$ normalized such that $h(0)=0$, $h(1)=1$ and 
$h (\infty)=\infty$.

Moreover, $h$ is a quasi-conformal homeomorphism 
on compact subsets of $\CC -K$ and $h(K)$ is also 
a thin Cantor set.

Also if we truncate $\mu$ in an $\eps_n$ neighborhood
of $K\cup \{\infty \}$, $\mu_n =\mu . {\hbox {\bf 1}}_{\CC- V_{\eps_n}
(K\cup \{ \infty \})}$, and we consider the classical 
rectifications $(h_n)$ of $(\mu_n) $, then $h_n \to h$
uniformly on the Riemann sphere.
}
\end{theorem}

\begin{definition}\textbf{(Generalized rectifications).}
{The solutions of the 
Beltrami equation provided by theorem \ref{rectif1} are  
called generalized rectifications.}
\end{definition}

A similar proof as for theorem \ref{rectif1} gives

\begin{theorem}{Let $(\mu_n)$ be a sequence
of Beltrami forms  
with $||\mu_n ||_{L^\infty (C)} \leq k(C) <1$ for 
compact subsets $C\in \CC-K$
such that 
$$
\lim_{n\to +\infty} \mu_n (z) =\mu (z)
$$
for almost every $z \in {\overline {\CC}}$, where
$\mu$ is a Beltrami form as in theorem \ref{rectif2}.
Let $(h_n)$ be the classical Morrey-Ahlfors-Bers 
rectifications normalized in the usual way, and $h$ 
the generalized rectification for $\mu$.
Then we have uniformly on the Riemann sphere 
$$
\lim_{n\to +\infty } h_n =h .
$$
}
\end{theorem}

\subsection{ Rectification of potential conformal structures.}\label{rectif_pot}

The "effective" version of the rectification theorem in the 
previous section are more explicit for potential 
conformal structures that are allowed to be even not locally 
quasi-conformal.

From now on we fix $K$ to be a totally thin Cantor
set with all points regular for the Dirichlet problem.

Let $\xi$ be a potential conformal. To each annulus $A$ in 
the analytic tree of $K$, there corresponds by
$E^{-1} \circ \varphi_K$  a finite number of 
rectangles $(R_j(A))$
with horizontal and vertical sides in 
the log-B\"ottcher upper half plane  (the vertical 
sides are contained in the preimage of the 
skeleton $E^{-1} (\G' (K) )$).
There are  natural identification of the lateral sides 
of the preimage of the skeleton 
$ E^{-1} (\G' (K) )$ defined by $\varphi_K$.
Pasting in the prescribed way the vertical boundaries of 
the rectangles $(R_j (A))$ we recover the annulus $A$, and 
we observe that t is easy to 
check that i
$$
\mod \  (A) =\sum_j R_j (A).
$$

We define the modulus of $A$ with respect to $\xi$
by
$$
\mod_{\xi} A =\sum_j \mod_{\xi'} R_j (A) =
\sum_j \mod L (R_j (A)).
$$t is easy to 
check that i
Observe that when $\xi$ is quasi-conformal with 
respect to $\s_0$ in a neighborhood of $A$ this 
definition coincides with the one given before 
(observe that $L$ defines a local rectification 
of each $R_j (A)$ that respects the vertical gluing).

\begin{definition}\textbf {(Admissible potential conformal 
structures).} A potential conformal structure $\xi$ 
is admissible if the analytic tree $\cT_\o (K)$ is 
thin for the conformal structure $\xi$. 
This means that for any infinite branch $(A_i)$ deprived from 
the root
$$
\sum_i \mod_\xi A_i =+\infty
$$
(note that $\mod_\xi A_0 =+\infty$ is automatic).
\end{definition}

\bigskip

{\bf Observation.}

We remind that an admissible potential conformal structure
is not necessarily admissible in the sense of the definition
in theorem \ref{rectif2} because it can be 
non-quasi-conformal on compact subsets of $\CC-K$.
Thus we allow even more degenerate conformal structures.
On the other hand the ellipse field defining $\xi$ is 
of a very special nature: All ellipses have principal
axes tangent or orthogonal to equipotentials.

\bigskip

Note that a potential conformal structure $\xi$  
has a well defined Beltrami form $\mu$ because
$L$ is almost everywhere differentiable.

\begin{theorem}{

Let $\xi$ be an  
admissible potential conformal structure with Beltrami form
$\mu$.

There exists a unique absolutely continuous 
almost everywhere differentiable 
homeomorphism rectifying the conformal structure $\xi$,
$h : ({\overline {\CC}}, \xi ) \to ({\overline {\CC}},
\s_0)$, i.e. solving the Beltrami equation
$$
\bar \partial h =\mu \ \partial h,
$$
with $h$ normalized such that $h(0)=0$, $h(1)=1$ and 
$h (\infty)=\infty$.

Moreover, $h(K)$ is a totally thin Cantor set.
Also if $\xi_n$ are potential conformal structures with 
corresponding $(d_n)$ and $(k_n)$ lipschitz maps 
(so that $\xi_n$ is quasi-conformal) such that $d_n \to d$
and $k_n \to k$ uniformly, then uniformly on ${\overline {\CC}}$,
$$
\lim_{n\to +\infty} h_n =h
$$
where the $(h_n)$ are the classical rectifications normalized
as usual.

The rectification $h$ is called the generalized rectification
of $\xi$. 
}
\end{theorem}

\begin{proof} {The proof follows the same lines as the one of theorem \ref{rectif1}.
We consider sequences $(d_n)$ and $(k_n)$ 
of Lipschitz homeomorphisms
$d_n : \TT \to \TT$ and $k_n : \RR_+ \to \RR_+$,  that 
converge uniformly to $d$ and $k$ respectively. 
We consider the corresponding 
conformal structures 
$$
\xi_n = (\varphi_K^{-1} \circ E)_* \ L_n^* \ \s_0 \ ,
$$
where $L_n (x,y)=(d_n (x) , k_n (y))$.

The conformal structures $\xi_n$ are quasi-conformal
with respect to $\s_0$. As in theorem \ref{rectif1} 
we prove that the sequence 
of normalized rectifications $h_{\xi_n} : 
({\overline {\CC}}, \xi_n ) \to ({\overline {\CC}},
\s_0)$ converges to a unique mapping $h$ that 
satisfies the conditions of the theorem. We have a 
similar proof. First the sequence of $(h_{\xi_n} (K))$ 
is uniformly bounded. 

This sequence has a limit $K_\infty$ in 
Hausdorff topology because any limit compact is a totally 
thin Cantor set and its analytic tree is determined by $\xi$,
(more precisely $d$ and $k$ determine the modular invariants
and $d$ alone determines the angular invariants) 
and the rectifications are
normalized.

Now, on the exterior of $K$ every point is determined
by its potential $G_K (z)$ and its set of external angles 
$\t_K (z)$ (in general there is only one external angle 
but it can have more 
than one when it belongs to a critical 
external ray). We observe that the Green map of 
$h_{\xi_n} (K)$ converges uniformly on compact sets
to the Green map of $K_\infty$. 
The rectifications $h_{\xi_n}$ transport
equipotentials (resp. external rays) of $K$ into 
equipotentials (resp. external rays) of $h_{\xi_n} (K)$
according to the map $k_n$ (resp. $d_n$). Moreover the 
Green function of $h_{\xi_n} (K)$ converges uniformly 
on compact sets of $\CC -K$ to the 
Green function of $K_\infty$. It follows that $h_{\xi_n}$
converges uniformly on compact sets of $\CC-K$ to a
map $h$ that maps a point $z\in \CC -K$ to the 
point $h(z)\in \CC-K_\infty$ determined by 
\begin{align*}
G_{K_\infty} (h(z))=k (G_K (z)) ,\\
\t_{K_\infty} (h(z))=d (\t_K (z)) , \\
\end{align*}
The map $h$ extends continuously to the Cantor set $K$
and $h(K) =K_\infty$.

We finally observe that on the complement  of $K$ and 
$h(K)$, the map $h$ is an absolutely continuous 
almost everywhere differentiable homeomorphism
(this map is just $L$ expressed in the B\"ottcher 
coordinates of the domain and the range).

Again the uniform convergence on $K$ is obtained from 
the convergence of finite parts of the analytic tree 
to the analytic tree of the limit Cantor set.}
\end{proof}

We observe that in the precedent proof we didn't 
make essential use of the fact that the limit mapping
on external angles $d$ was a homeomorphism. It was 
only used to define the modulus of annulus with 
respect to $\xi$ in order to identify the analytic 
tree of the limit. This is the key point observation 
before going to the most general rectification theorem.

\subsection { Rectification of virtual conformal structures.} \label{rectif_virt}

We consider a totally thin Cantor set $K$ and a 
virtual conformal structure $\xi$ on 
${\overline {\CC}}$ given by $L$ (or $(d,k)$). We denote by $\xi'$ the corresponding 
conformal structure on $\HH /\ZZ$ defined 
by the same $L$. 
The Cantor set $K$ is the ideal boundary of the 
Riemann surface ${\overline {\CC}}-K$. The measure 
$\mu_d$ being non-atomic and the critical external 
rays being countable, there is a natural way of 
projecting $\mu_d$ into a non-atomic probability 
measure of the ideal boundary. The  support of this
projected measure is a Cantor set $K_0\subset K$.

Let $R$ be a rectangle in the log-B\"ottcher upper 
half cylinder $\HH /\ZZ$ having vertical and horizontal
sides.
It is then natural to define 
by analogy with a classical complex structure
$$
\mod_{\xi '} R =\mod L (R) ,
$$
and 
$$
\mod_\xi A =\sum_j \mod_{\xi '} R_j (A) ,
$$
for an annulus $A$ in the analytic tree of $K$. 
Note that this definition is compatible with the previous 
ones when the virtual conformal structure is a classical 
conformal structure. Observe also that if $J_0(A)\subset \TT$ denotes 
the set of external angles whose corresponding external 
ray intersects $A$ and if $J_1 (A) \subset \RR_+$ is the interval 
of equipotential values intersecting $A$, then 
$$
\mod_\xi A =\frac{|J_0(A)|}{ \mu_d (J_0(A))} \frac{ |k (J_1 (A))|}{|J_1 (A)|} \mod A ,
$$
where $\mu_d$ is the non-atomic probability measure 
corresponding to $d\in \hat {\cC}_0 (\TT )$, 
and $|J(A)|$ denotes
the Lebesgue (linear) measure of $J(A)$. Thus 
we observe that if $\supp \mu_d \not= \TT$ (i.e. $d$ 
is not a homeomorphism) then we may 
have for some annulus $\mod_\xi A =+\infty$. This is 
a new feature that did not happen for non virtual 
conformal structures. Also from the above formula it 
follows that we cannot have $\mod_\xi A =+\infty$ for 
all annulus $A$ at a given depth. Otherwise by additivity 
of $\mu_d$ it would follow $\mu_d (\TT) =0$ contradicting 
that $\mu_d$ is a probability measure.

\begin{definition} {A virtual conformal 
structure $\xi$ is admissible if 
for each infinite 
branch $(A_i)$ deprived from the root  of the 
analytic tree $\cT _\o (K)$ we have 
$$
\sum_i \mod_\xi A_i =+\infty.
$$
}
\end{definition}

We have the following rectification theorem.

\begin{theorem}\label{rectif3}{Let $\xi$ be a virtual 
conformal structure in ${\overline {\CC}}-K$ 
(where $K$ is a totally thin Cantor set as above) 
defined by 
$d\in \hat {\cC}_0 (\TT)$ and $k\in {\Homeo}_+ 
(\RR_+ )$.
We assume that $\xi$ is admissible and that $0,1 \in K_0$ 
where $K_0$ is the support of the projection of $\mu_d$ on 
$K$.

There is a unique  
map correspondence $h\in \hat {\cC}_0 (
{\overline {\CC}}$ which 
rectifies $\xi$ such that $h(0)=0$, $h(1)=1$ and $h(\infty )=
\infty$. 

More precisely, the subset $K_\infty$ 
of regular (or non-isolated) points of  
$h(K)$ is a totally thin Cantor set, $K_\infty =h(K_0)$
and on the 
complement of $E^{-1} (\G'(K))$,
$$
 h\circ \varphi_K^{-1} \circ E  = 
\varphi_{h(K_0)}^{-1} \circ E\circ L.
$$

If there is no annulus $A$ in the analytic tree that 
degenerates (that is $\mod_\xi A < +\infty$, i.e. the 
support of $\mu_d$ is $\TT$ so $K_0=K$), then 
$h$ is a homeomorphism and $h(K)=K_\infty$.

Moreover, if $(\xi_n)$ is a sequence of 
potential conformal 
structures with 
associate homeomorphisms $(d_n)$ and $k_n=k$ with $d_n \to d$
uniformly on $\TT$,
then uniformly on the Riemann sphere 
$$
\lim_{n\to +\infty } h_{\xi_n } =h,
$$
where the rectifications have been properly normalized
composing by a suitable Mo\"ebius transformation.
}
\end{theorem}

Before going into the proof, we have to study the 
limit Cantor set $K_\infty$ and more precisely how its
analytic tree is related with the one of $K$.

\subsubsection {Combinatorial collapsing.} \label{collapsing}

We consider the above situation. The virtual conformal 
structure $\xi$ is determined by $(d,k)$. 
We construct the analytic tree $\tau_\o^\xi (K)$ of $K$ 
with respect to $\xi$ from $\tau_\o (K)$ as follows.

We first delete from $\tau_\o (K)$ all branches after 
a vertex $A$ such that $\mod_\xi A =+\infty$. Given 
a vertex $B$, we denote by $\a_+^{(1)} (B), 
\a_+^{(2)}(B) \in \TT$ the two external 
angles of the critical point in the outer boundary of $B$,
and $\a_-^{(1)}(B), \a_-^{(2)}(B) \in \TT$ the 
two angles of the critical point in the inner boundary.
We have that 
$$
\left \{ \frac{\a_+^{(1)} -\a_-^{(1)}}{|\a_+^{(1)} -\a_+^{(2)} |}, 
\frac{\a_+^{(1)} -\a_-^{(2)}}{|\a_+^{(1)} -\a_+^{(2)}| } \right \}
$$
is the angular invariant of $B$ (the expression
$|\a_+^{(1)} -\a_+^{(2)}|$ is the harmonic measure of 
the annulus $B$). ic
Observe that a vertex $A$ satisfies $\mod_\xi A =+\infty$
if and only if 
$$
d (\a_+^{(1)} (A) )=d (\a_+^{(2)} (A) ),
$$
or equivalently, if and only if for its parent $B$ we have 
$$
d (\a_-^{(1)} (B))=d (\a_-^{(2)} (B) ).
$$ 

Next we remove 
all vertices in the remaining tree that have only one children.
The new branches of the tree can be composed by several 
old branches. We change the modular invariant of the father 
vertex (denote it by $A_1$) by the sum of 
its old modular invariant with the 
modular invariants of the deleted old children (denote them
by $A_2 , \ldots , A_n$).
So
$$
M^\xi (A_1) =\sum_{i=1}^n \mod A_i.
$$
 
Its new angular invariants will 
be 
\begin{align*}
\t_1^\xi (A_1) &=\frac{1}{ |d (\a_+^{(1)} (A_1))-d (\a_+^{(2)} (A_1))|}
\sum_{i=1}^{n} d (\a_+^{(1)} (A_i )) -d (\a_-^{(1)} (A_i)) \\
&=\frac{d(\a_+^{(1)} (A_1))-d(\a_+^{(1)} (A_n))}
{|d (\a_+^{(1)} (A_1))-d (\a_+^{(2)} (A_1))|} \\
\t_2^\xi (A_1) &=\frac{1}{|d (\a_+^{(1)} (A_1))-d (\a_+^{(2)} (A_1))|}
\sum_{i=1}^n d (\a_+^{(1)} (A_i )) 
-d (\a_-^{(2)} (A_i)) \\
&=\frac{d(\a_+^{(1)} (A_1))-d(\a_+^{(1)} (A_n))}
{|d (\a_+^{(1)} (A_1))-d (\a_+^{(2)} (A_1))|} \\
\t^\xi (A_1) &= \{ \t_1^\xi (A_1), \t_2^\xi (A_2) \} \\
\end{align*}

We denote by $(\hat A_i)$ the final tree. It is a binary tree (we cannot remove all vertices 
of an infinite branch of the original tree because the 
measure $\mu_d$ has no atoms). The new analytic tree is
$$
\tau^\xi_\o (K) =\left ( \hat A_i , M^\xi (\hat A_i ), 
\t^\xi (\hat A_i )\right ).
$$

\textbf {Proof of Theorem \ref{rectif3}.}

Consider a sequence of classical 
conformal structures
$(\xi_n )$ quasi-conformal with respect to 
$\s_0$ (just take $d_n$ and $k_n$ to be lipschitz)
as in the Theorem.

Observe that on $E^{-1} (\G' (K))$ we 
have,
$$
 h_{\xi_n} \circ \varphi_K^{-1} \circ E  = 
\varphi_{h_{\xi_n} (K)}^{-1} \circ E\circ L_n.
$$
As before, the mappings $h_{\xi_n}$ being normalized,
the sequence of compacts sets $(K_n)$, $K_n =h_{\xi_n} (K)$, is 
uniformly bounded. We can extract converging sub-sequences
\begin{align*}
\lim_{k\to +\infty} K_{n_k} &=K'_\infty ,\\
\lim_{k\to +\infty} \varphi_{K_{n_k}} &= \varphi_\infty ,\\
\lim_{k\to +\infty} \G (K_{n_k}) &= \G_\infty (K).\\
\end{align*}

The map $\varphi_\infty^{-1}$ is holomorphic near $\infty$ and 
is defined in the kernel of the domain of definition of the 
$(\varphi_{K_n}^{-1})$,
$$
\varphi_\infty^{-1} : {\overline {\CC}}-E\circ L\circ E^{-1}
(\G'(K)) \to 
{\overline {\CC}}- (K'_\infty \cup \G_\infty (K)).
$$

When there is collapsing (i.e. when $d$ is not a homeomorphism)
the limit $K'_\infty$ contains non-regular points. More
precisely, to each collapsing annulus $A$ there corresponds 
an isolated point of $K'_\infty$ which the limit of all points of 
$K$ enclosed by $A$. This point if the image of a tip of 
a segment of $E\circ L\circ E^{-1} (\G'(K))$. The map 
$\varphi_\infty^{-1}$ extends holomorphically at this point
in a univalent way. We denote by $K_\infty$ the subset of 
regular points of $K'_\infty$ (i.e. $K'_\infty$ deprived of 
its isolated points). The extension 
of $\varphi_\infty$ constructed in this way is the Green map
of $K_\infty$. It is not difficult to see that 
the analytic tree of $K_\infty$ coincides
with the analytic tree $\tau_\o^\xi (K)$ of $K$ with 
respect to the virtual conformal structure $\xi$ constructed
above. Thus $K_\infty$ is a totally thin Cantor set 
uniquely determined from $K$ and $\xi$. Thus we don't
need to extract converging sub-sequences: The limit
$K'_\infty$ is 
uniquely determined. The limit of the mappings 
$(h_{\xi_n})$ is uniquely determined on $K$, but 
also outside $K$ by the potential outside of $K_\infty$
(as in the proof of theorem \ref{rectif1}).
Passing to the 
limit in the above relation we obtain on $E^{-1} (\G'(K))$,
$$
 h\circ \varphi_K^{-1} \circ E  = 
\varphi_{K_\infty}^{-1} \circ E\circ L.
$$
The other properties of the theorem are obtained as 
in theorem \ref{rectif2}.

\subsection {Generalized rectifications for continua.} \label{continua}

In this section we consider a {\it connected} compact set 
$K$ as above (so $\G_K =\emptyset$ and $\varphi_K$ is 
a conformal representation).

\begin{definition} \textbf{(Green equivalence in $\HH/\ZZ$).} \label{green_equivalent}
{Two conformal structures $\xi$ and $\eta$ on 
$\HH /\ZZ$ are Green equivalent if
there exists an absolutely continuous almost 
everywhere differentiable homeomorphism $L : \HH /\ZZ 
\to \HH /\ZZ$ of the form 
$$
L(x,y) = (x, h(y)),
$$
such that 
$$
\eta =L_* \xi.
$$
In that case $L$ extends into a homeomorphism
of ${\overline {\HH}}$ by the identity on $\TT=\RR /\ZZ$. 
The map $L$ is an absolutely continuous almost everywhere
differentiable homeomorphism if and only if  
$h :\RR_+ \to \RR_+$ is an absolutely 
continuous homeomorphism.}
\end{definition}

This is a well 
defined equivalence relation since the class of homeomorphisms $L$ considered is a composition subgroup.
Observe that $L$ leaves globally invariant vertical 
lines, i.e. Green lines for the potential $z\mapsto
\Im z$ in $\HH /\ZZ$.

We recall that we denote  $E(z)=e^{-2\pi i z}$.

\begin{definition} \textbf{(Green equivalence outside 
a compact connected set $K$).}{
Let $K\subset  {\CC}$ be a compact 
connected set as above.
Two conformal structures $\xi$ and $\eta$ on $\O=
{\overline {\CC}}-K$ are Green equivalent 
outside $K$ if  
\begin{align*}
\xi' &=(\varphi_K^{-1} \circ E)^* \xi ,\\
\eta' &=(\varphi_K^{-1} \circ E)^* \eta ,\\
\end{align*}
are Green equivalent on $\HH /\ZZ$.

Then there exists an absolutely continuous 
almost everywhere differentiable homeomorphism 
$l: {\overline {\CC}} -K \to {\overline {\CC}} -K$,
such that 
$$
\eta =l_* \xi,
$$
and $l$ leaves invariant Green lines, maps equipotentials
into equipotentials, and is
defined by
$$
l\circ \varphi_K^{-1} \circ E 
=\varphi_K^{-1}\circ E \circ L ,
$$
where $L:\HH \to \HH$ is the mapping realizing the 
equipotential equivalence of $\xi'$ and $\eta'$.
}
\end{definition}

The following result shows how are related the 
Morrey-Ahlfors-Bers rectification of two Green 
equivalent conformal structures that are quasi-conformal
(with respect to the standard complex structure $\s_0$).

\begin{theorem}\label{thm_green}{
Let $K\subset {\overline {\CC}}$ a compact connected
set as above. Assume that $0\in K$ and $1\in K$. 
Let $\xi$ and 
$\eta$ two conformal structures on $\O ={\overline {\CC}}
-K$ which are quasi-conformal with respect to $\s_0$ and 
Green equivalent outside $K$. Then the mapping 
$l$ is quasi-conformal (i.e. uniformly quasi-conformal 
on ${\overline {\CC}}-K$) 
and we can extend $\xi$ and 
$\eta$ into quasi-conformal structures 
on ${\overline {\CC}}$ defining 
$$
\xi_{/K} =\eta_{/K} =\s_{0 \ /K}.
$$

Let $h_{\xi} :({\overline {\CC}}, \xi) \to 
({\overline {\CC}}, \s_0)$ and 
$h_{\eta} :({\overline {\CC}}, \eta) \to 
({\overline {\CC}}, \s_0)$ be the unique 
Morrey-Ahlfors-Bers rectifications normalized so 
that 
\begin{align*}
h_{\xi} (0)=0 , \ 
h_{\xi} (1)=1 , \ h_{\xi}(\infty ) = \infty  ,\\
h_{\eta} (0) = 0 , \ h_{\eta} (1) =1 , \ 
h_{\eta }(\infty ) = \infty ,\\
\end{align*}
Then we have 
\begin{align*}
h_{\xi /K} &=h_{\eta \ /K} , \\
h_{\xi /{\overline {\CC}}-K} &=
h_{\eta /{\overline {\CC}}-K} \circ l .\\
\end{align*}
In particular we have
$$
h_{\xi} (K)=h_{\eta} (K).
$$
}
\end{theorem}

\begin{lemma}\label{lemma_poincare}{If $L: \HH /\ZZ \to \HH /\ZZ$, 
$(x,y) \mapsto 
(x, h(y))$ is a $C$-quasi-conformal homeomorphism
then $h$ (and also $L$) 
is $C$-bilipschitz, i.e. for $y,y'\geq 0$,
$$
C^{-1} |y-y'| \leq |h(y)-h(y')| \leq C |y-y'|.
$$
In particular, if $l$ is related to $L$ as in 
definition \ref{green_equivalent}, if $d_P$ is the Poincar\'e distance 
in ${\CC}-K$, then $l$ is bounded away
from the identity for the uniform norm for $d_P$, more
precisely,for 
$z\in \CC-K$,
$$
d_P (l(z) ,z) \leq \log C.
$$
}
\end{lemma}

\begin{proof}\textbf {(Lemma \ref{lemma_poincare}).}{The map $L$ is 
differentiable almost everywhere and 
we compute
$$
DL =\begin{pmatrix} 1 & 0 \\ 0 & h'(y) \\ \end{pmatrix}.
$$
Thus for almost every $y>0$, $|h'(y)| \leq C$ and the 
first statement follows. 
For the last statement we observe that it is enough to check
that $L-\id$ is bounded for the Poincar\'e distance 
$d_{\HH /\ZZ}$ of $\HH /\ZZ$. But 
$$
d_{\HH /\ZZ} (L(z),z)=\left | \log \left 
( \frac{h(y)}{y} \right ) \right | \leq \log C .
$$}
\end{proof}

\begin{proof} \textbf {(Theorem \ref{thm_green}).}
{Given $h_{\eta}$ we can define $h: {\overline {\CC}} 
\to {\overline {\CC}}$ by 
\begin{align*}
h_{/K} &=h_{\eta \ /K} , \\
h_{/{\overline {\CC}}-K} &=
h_{\eta \ /{\overline {\CC}}-K} \circ l .\\
\end{align*}

If $z_n\to z_0 \in K$ with 
$z_n \in {\overline {\CC}}-K$ then $l(z_n)-z_n \to 0$
because from lemma \ref{lemma_poincare} we know that $d_P (l(z_n) , z_n)$ is 
bounded. So $l(z_n) \to z_0$ and $l$ extends continuously 
to $K$ by the identity. Thus the map $h$ is a homeomorphism.
Moreover $h$ is quasi-conformal on ${\overline {\CC}}-K$
because $l$ is quasi-conformal there, and coincides with the  
global quasi-conformal homeomorphism $h_{\eta}$ in $K$.
Now Rickman's theorem (see Theorem \ref{thm_rickman}) 
implies that $h$ is quasi-conformal
and on $K$,
$$
\bar \partial h =\bar \partial h_{\eta} =0.
$$

So $h$ rectifies $\xi$ as is normalized as $h_{\xi}$. 
By uniqueness of Morrey-Ahlfors-Bers theorem we have 
$h_{\xi}=h$.}
\end{proof}

A corollary of this result justifies the terminology used:

\begin{corollary}{If $\xi$ and $\eta$ are Green 
equivalent as above, 
if $\eta$ happens to be the standard complex structure $\s_0$ 
then $h_{\xi} (K)=K$, 
and $h_{\xi}$ maps equipotentials 
of $K$ into equipotentials of $K$, and leaves 
globally invariant external rays of 
$K$.

Also if $h_{\eta}$ maps equipotentials of $K$ into 
equipotentials of $h_{\eta} (K)$, and 
external rays of $K$ into external rays of $h_\eta (K)$ 
then $h_{\xi}$ has 
the same property.
}
\end{corollary}

\begin{proof}{In the case that $\eta =\s_0$,
we have $h_{\eta}=\id$ and 
$h_{\xi /{\overline {\CC}}-K}=l$. The statement
follows from the fact that 
$\varphi_K\circ l\circ \varphi_K^{-1}$ 
maps circles centered at $0$ into concentric circles,
and radial lines into radial lines.

The second statement follows 
from $h_{\xi /{\overline {\CC}}
-K}=
h_{\eta /{\overline {\CC}}-K} \circ l$.}
\end{proof}

The last theorem suggest that when $\xi$ is Green  
equivalent to a quasi-conformal structure $\eta$ we 
can define a generalized rectification map using 
the rectification of $\eta$.

\begin{definition}\textbf {(Generalized rectification).}
{Let $\xi $ and $\eta$ be Green 
equivalent conformal structures 
on ${\overline {\CC}}-K$  with $\eta$ quasi-conformal
with respect to $\s_0$. We assume $0,1\in K$.
The generalized rectification of $\xi$
is $h_{\xi} : {\overline {\CC}} \to {\overline {\CC}}$
defined by 
\begin{align*}
h_{\xi /K} &=h_{\eta /K} , \\
h_{\xi /{\overline {\CC}}-K} &=
h_{\eta /{\overline {\CC}}-K} \circ l .\\
\end{align*}
Observe that $h_{\xi}$ is absolutely continuous but 
not necessarily almost everywhere differentiable on 
$K$ and not even continuous (!) because there is no 
reason that $l$ extends continuously by the identity 
on $K$ when $K$ is not locally connected. 
It is almost everywhere differentiable on 
${\overline {\CC}}-K$ and 
$$
(h_{\xi})_* \xi =\s_0.
$$
If $\partial K$ has measure $0$ then $h_{\xi}$ is 
almost everywhere differentiable, $(h_{\xi})_* \xi =\s_0$,
although $h_{\xi}$ may be discontinuous. When all 
external rays land at $K$ (in particular when $K$ is 
locally connected), the map $h_{\xi}$ is a 
homeomorphism of the Riemann sphere.
}
\end{definition}

We prove that the generalized rectification 
does not depend on the choice of $\eta$.

\medskip

\begin{remark}

\medskip

It is not difficult to construct non continuous generalized 
rectifications. Consider a continuum $K$ with a non landing
external ray $\g$. Choose a sequence of points $z_n\in \g$
with decreasing potential such that $|z_n -z_{n+1}|\geq \eps_0 >0$
for some $\eps_0 >0$. Then construct by interpolation an 
absolutely continuous homeomorphism $h:\RR_+ \to \RR_+$
such that $h (G_K (z_{n+1}))=G_K (z_n)$. The corresponding
map $l$ is the generalized rectification for $\xi =l^* \s _0$
in $\CC-K$ and does not extend continuously to the 
identity on $K$. But $h_\xi$ is the identity on $K$.

\end{remark}

\bigskip

The precedent corollary also holds for generalized 
rectifications 
(i.e. when $\xi$ is not assumed to be 
quasi-conformal) with the same proof.

The unsatisfactory part of the generalized rectification
is that in general it is not even a 
homeomorphism and also it is not 
almost everywhere differentiable.
We study this problem below and give some sufficient
conditions on the equipotential equivalence that 
implie the almost everywhere differentiability of 
the generalized rectification.

On the other hand, 
the generalized rectification can be regarded as a 
(highly singular) 
solution to a Beltrami equation with unbounded 
Beltrami form (i.e. $||\mu ||_{L^\infty} =1$).
Therefore it is satisfactory to have the following uniqueness 
result.

\begin{theorem} \textbf {(Uniqueness).}{If $\xi$ has a 
generalized rectification, then it is unique (i.e.
it does not depend on the choice of $\eta$).}
\end{theorem}

\begin{proof}{If $h$ and $h'$ are two generalized 
rectifications corresponding to two quasi-conformal
structures $\eta$ 
and $\eta'$, then $h'\circ h^{-1}$ is the identity
on $K$.
Moreover  on $\O ={\overline {\CC}}-K$ we have
$$
h'\circ h^{-1} = h_{\eta'} 
\circ (l'\circ l^{-1}) \circ h_{\eta}^{-1}.
$$
This is a composition of quasi-conformal mappings
on $\O$
because $l'\circ l^{-1}$ is quasi-conformal on $\O$ 
since it transports $\eta$ into $\eta'$. Also 
from lemma \ref{lemma_poincare} we get that the 
map $l'\circ l^{-1}$ is bounded from the identity 
in $\O$ for the Poincar\'e metric. Thus $h'\circ h^{-1}$
is continuous and it is a homeomorphism of the Riemann
sphere.
Now, $h'\circ h^{-1}$ is a homeomorphism, quasi-conformal
in ${\overline {\CC}}-K$, and coinciding with the 
identity on $K$.
We can use Rickman's theorem 
to conclude that $h'\circ h^{-1}$ is conformal, and 
then necessarily $h=h'$.}
\end{proof}

Now we have the following approximation theorem:

\begin{theorem} \label{thm_approx}{Let $\xi$ be a conformal structure
on ${\overline {\CC}} -K$ Green equivalent 
to a quasi-conformal structure $\eta$.

We assume that $\xi$ is quasi-conformal on compact 
subsets of $\CC -K$ (i.e. $l$ is quasi-conformal 
on compact subsets of $\CC-K$). Let $\mu$ be its
associated Beltrami form extended by $0$ on $K$.

Let $(\mu_n)_{n\geq 0}$ be a sequence of quasi-conformal
(i.e. $||\mu_n||_{L^\infty} <1$) Beltrami forms 
converging to $\mu$ almost everywhere
and such that $|\mu_n| \leq  |\mu |$. Let $\xi_n$ 
be the associated complex structure to $\mu_n$ and
$h_n : ({\overline {\CC}} , \xi_n) \to 
({\overline {\CC}}, \s_0 )$ be the Morrey-Ahlfors-Bers
rectification homeomorphism, normalized such that 
$h_n (0)=0$, $h_n (1)=1$ and $h_n (\infty )=\infty$.
We assume that the conformal structures $\xi_n$ are
Green equivalent to $\xi$. We made also the 
assumption that the modulus of an annulus 
bounded by $K$ and an equipotential of $K$ for the 
quasi-conformal structure $\xi_n$ is uniformly bounded 
on $n$ from above and away from $0$.
 
Then 
$$
\lim_{n\to +\infty} h_n =h_\xi
$$
uniformly on compact subsets of $\CC -K$.

Moreover, if the complex structures $\xi_n$ are 
Green equivalent to $\xi$ (such a 
sequence always exists from Proposition \ref{prop_approx}) 
then $h_n=h_\xi$ on $K$.}
\end{theorem}

\begin{lemma}{The family $(h_n)_{n\geq 0}$ is 
equicontinuous in $\CC -K$.}
\end{lemma}

\begin{proof}{The proof runs as the classical proof
of equicontinuity of families of uniformly quasi-conformal
homeomorphisms. We prove that the family $(h_n)_{n\geq 0}$ 
is uniformly H\"older on compact subsets of $\CC-K$.
Let $z_0 \in \CC -K$. Let $r_0 >0$ small enough 
such that
${\overline {B (z_0 , r_0)}}\subset \CC -K$
and $B(z_0 , r_0)$ does not contain one of the 
points $0$ or $1$, say $1$ for example 
(furthermore this disk does not 
contain $\infty$). Let $0<r <r_0$. The quasi-conformal 
mappings $h_n$ are $M$-quasi-conformal for a 
uniform $M$ on ${\overline {B(z_0 , r_0)}}$.
We consider the annulus
$$
A =B (z_0 , r_0) -{\overline {B(z_0 ,r)}}.
$$
We have 
$$
\mod h_n (A) \geq M^{-1} \mod A =M^{-1} \frac{1}{ 2 \pi}
\log \left ( \frac{r_0}{ r} \right ).
$$
The component of the complement of $A$ not containing 
$h_n (z_0)$ contains $1$ and $\infty$.
It follows from Teichmuller estimate ([Le-Vi]) that 
there exists a universal constant $C>0$ such 
that in the spherical metric
\begin{align*}
\diam h_n ({\overline {B (z_0 , r)}}) &\leq 
C e^{-2\pi \mod (h_n (A))} \\
&\leq C e^{M^{-1} \log \left ( \frac{r_0}{ r} \right )}=
C \left ( \frac{r_0}{ r} \right )^{M^{-1}}.\\
\end{align*}

This proves that the family $(h_n)_{n\geq 0}$ is uniformly 
H\"older on compact subsets of $\CC -K$.}
\end{proof}

\begin{proof}
{\textbf{(Theorem \ref{thm_approx})}}
{ Let $h$ be an accumulation point
of the family $(h_n)_{n\geq 0}$ in $\CC -K$,
$$
\lim_{k\to +\infty} h_{n_k} =h.
$$
The assumption on uniformity of the modulus of annulus
bounded by $K$ and an equipotential implies that 
the limit $h$ is not the infinite constant and that
$h(\CC -K )=\CC -K$.
As in the classical situation, the uniform quasi-conformality
on compact subsets of $\CC-K$ proves that such a limit 
$h$ is a homeomorphism of $\CC-K$ (see [Le-Vi]). Here 
the Green equivalence of $\xi_n$ to $\xi$ proves 
directly this fact.
Let $l_n =h_\eta^{-1} \circ h_n$. Then 
$\lim_{k\to +\infty} l_{n_k} =\hat l =h_\eta^{-1} \circ h$.
By construction $\hat l_* \xi =\eta$, so 
$(l\circ \hat l^{-1} )_* \eta =\eta$.
Now $l\circ \hat l^{-1} (\infty )=\infty$ and 
$l\circ \hat l^{-1} $ is a conformal automorphism
of $({\overline {\CC}}-K , \eta)$ whose angular action 
at $\infty$ is the identity. So $l=\hat l$ and 
$h=h_\eta \circ \hat l=h_\eta \circ l =h_\xi$.}
\end{proof}

\begin{proposition} \label{prop_approx}
{Given a conformal structure 
$\xi$ as in theorem \ref{thm_approx}, there exists a sequence 
of Green equivalent 
quasi-conformal conformal structures $(\xi_n)_{n\geq 0}$
such that the corresponding sequence of Beltrami forms 
$(\mu_n)_{n\geq 0}$ converge to $\mu_\xi$ 
almost everywhere and for $z\in \CC-K$,
$$
|\mu_n (z)| \leq |\mu (z) |.
$$
}
\end{proposition}

\begin{proof}{The conformal structure $\xi$ is 
Green equivalent to a quasi-conformal 
structure $\eta$. Let $L(x,y)=(x,h(y))$ be the 
corresponding mapping in the upper half plane.
The continuous strictly increasing map 
$h$ is almost everywhere
differentiable.
The differential of $L$ for almost 
all points $(x,y)$ 
is
$$
DL (x,y) =\begin{pmatrix} 1 & 0 \\ 0 & h'(y) \end{pmatrix} \ .
$$
So the Beltrami form of $\mu_\xi$ at the point 
$z\in \CC -K$
corresponding to the point $(x,y) \in \HH$
is
$$
\mu_\xi (z) =\frac{\bar \partial L}{\partial L} =
\frac{1 -h'(y)}{1+h'(y)} .
$$

We define $\varphi_n =\min (n ,h')$.
The sequence of functions $h_n (x) =\int_0^x \varphi_n (u) du$  
converges pointwise to $h$, and the mappings 
$L_n (x,y)=(x, h_n (y))$ define complex structures $\xi_n$ 
in ${\overline {\CC}}-K$ that satisfy the required 
properties.}
\end{proof}

\bigskip

Theorem \ref{thm_approx} together with Proposition 
\ref{prop_approx} show that if the generalized rectification is 
quasi-conformal on compact subsets of $\CC -K$
then it can be obtained as uniform limit on compact 
subsets of $\CC -K$ of quasi-conformal homeomorphisms
of the Riemann sphere.
Note that this result provides an alternative way 
to define the generalized rectification.

The problem of almost everywhere differentiability 
of the generalized rectification 
$h_{\xi}$ is equivalent to the problem of 
extension of $l$ to ${\overline {\CC}}$ into 
an almost everywhere differentiable homeomorphism.
The map $l$ extends radially to the identity 
for almost all $z\in K$ for the harmonic measure.
If $K$ is locally connected $l$ extend continuously 
to the identity on $K$ into a global homeomorphism
of the Riemann sphere. But we are looking for results that 
are independent of the structure of $K$.

In the following theorems we show that just 
a $C^0$ control on $l-\id$ on $\O$ 
for the Poincar\'e metric implies a differentiability 
result.

\begin{theorem}{Let $K$ be a full compact connected 
set in ${\overline {\CC}}$ with at least two points. 
If $l :{\overline {\CC}}
\to {\overline {\CC}}$ is a homeomorphism such that:

\item {$\bullet$} $l_{/K}=\id_K$ (so $l(K)=K$),

\item {$\bullet$} The map $z\mapsto 
d_P (l(z) , z)$ converges to $0$ when 
$z\to K$, $z\in {\overline {\CC}}-K$ 
(we denote by $d_P$ the Poincar\'e distance 
of ${\overline {\CC}}-K$).

Then $l$ is differentiable on all points of $K$ and 
$$
Dl_{/K} =\id.
$$
}
\end{theorem}

\begin{proof}{Let $z_0\in K$ and $h\in \CC$, $h\not= 0$.
If $z_0+h \in K$ then 
$$
l(z_0+h)-l(z_0) =(z_0+h)-z_0=h.
$$
We consider the case $z_0+h \notin K$.

Let $\eps>0$. 
There exists $\d >0$ such that for any $h'\in \CC$,
$|h'|< \d$, we have $d_P (l(z_0+h'), z_0+h') \leq \eps$.
Now the Poincar\'e infinitesimal length is 
related to the Euclidean infinitesimal length by 
$$
\frac{1}{2} \frac{|dz|}{\d (z)} \leq ds_P (z) \leq 2 
\frac{dz}{ \d (z)},
$$
where $\d (z) =d_E (z, K)=\min_{w\in K} |z-w|$.
For $\eps $ small, if $|h| \leq \d$ then 
$[z_0+h, l(z_0+h)] \subset {\overline {\CC}}-K$
and 
$$
\d_0 =\max_{w\in [z_0+h, l(z_0+h)]} \ d_E (w,K) \leq C |h|,
$$
because $d_E (z_0+h, K) \leq  |h|$ and $|l(z_0+h) -(z_0 +h)|
\leq 2\d_0 \eps$ (we can put $C=(1-2\eps )^{-1}$.)

So we obtain,
$$
\frac{1}{2} \frac{1}{\d_0} |l(z_0+h) -(z_0+h)|\leq 
d_P(l(z_0+h), z_0+h) \leq \eps.
$$
And for $|h|\leq \d$,
$$
|l(z_0 +h)-z_0 -h|\leq 2\eps \d_0 \leq 2C \eps |h|,
$$
and
$$
\left | \frac{l(z_0+h) -z_0}{ h} -1\right | \leq 2C \eps.
$$
That means that $Dl (z_0) =\id$.}
\end{proof}

We can get a better result for almost everywhere
differentiability. The next theorem is such a result.
 The proof of the next theorem 
gives more: We can balance how far is
$l$ from the identity in the Poincar\'e metric to the porosity
of $K$. We leave other more precise statement for a future version.

\begin{theorem}{
Let $K$ be a full compact connected 
set in ${\overline {\CC}}$. If $l :{\overline {\CC}}
\to {\overline {\CC}}$ is a homeomorphism such that:

\item {$\bullet$} $l_{/K}=\id_K$ (so $l(K)=K$),

\item {$\bullet$} For $z\in {\overline {\CC}}$ 
the map $z\mapsto d_P (l(z),z)$ is 
bounded ($d_P$ is the Poincar\'e distance). Let 
$M>0$ be an upper bound.

Then $l$ is differentiable almost everywhere on $K$ and,
more precisely,  
for all Lebesgue density points  $z\in K$,
$$
D_z l =\id.
$$
}
\end{theorem}

\begin{proof}
{By Lebesgue density theorem, for a.e. $z\in K$,
$$
\lim_{r\to 0} \frac{\lambda (B(z,r)\cap K)}{
\lambda (B(z,r))} =1.
$$
Thus for a set $K_0\subset K$ of full measure,
for any $z\in K_0$ there exists an increasing function
$\eps_z (r)>0$, such that 
$$
\lim_{r\to 0} \eps_z(r) =0,
$$
and 
$$
\frac{\lambda (B(z,r)\cap K)}{ 
\lambda (B(z,r))} \geq 1-\eps_z(r).
$$
We carry out the same proof as before for $z_0\in K_0$.
Given $\eps >0$ we choose $r_0>0$ such that 
$$
2 \eps_z(r_0 )^{1/2} M \leq \eps /2.
$$
Now if $r\leq r_0$, for any $w\in B (z, r/2)$, 
we have 
$$
\d (w) \leq r \eps_z(r )^{1/2},
$$
because $B(w, r \eps_z(r )^{1/2})$ must intersect
$K$  or we would have a too big hole in $B(z,r)$
incompatible with the Lebesgue density condition
at this scale.
Thus if $|h|=r/2<r_0/2$, 
$$
\d (z_0 +h) \leq 2 |h|  \eps_z(r )^{1/2}\leq 
2 |h|  \eps_z(r_0 )^{1/2}.
$$
So 
\begin{align*}
\frac{1}{ 2} \frac{1}{\d (z_0 +h) } |l(z_0 +h) -(z_0+h)|
&\leq d_P (l(z_0+h), z_0+h), \\
|l(z_0 +h) -(z_0+h)| &\leq 4 |h| \eps_z(r_0 )^{1/2} M, \\
\left | \frac{l(z_0 +h) -z_0}{ h}-1 \right | 
&\leq 4 \eps_z(r_0 )^{1/2} M, \\
\left | \frac{l(z_0 +h) -z_0}{ h}-1 \right | 
&\leq \eps. \\
\end{align*}
}
\end{proof}

\end{document}